\documentclass[11pt]{article}

 \usepackage{amsmath,amssymb,amscd,amsthm,esint}

\usepackage{graphics,amsmath,amssymb,amsthm,mathrsfs,color}

\usepackage{amsfonts}
\usepackage{esint}
\usepackage{pifont}
\usepackage{bbding}
\usepackage{latexsym}
\usepackage{mathrsfs}
\usepackage{amsmath}
\usepackage{amsthm}
\usepackage{amssymb}
\usepackage{graphicx}
\usepackage{lineno}
\usepackage{amssymb}
\usepackage{color}
\usepackage[left=1in, right=1in, bottom=1.4in]{geometry}

\newtheorem{theorem}{Theorem}[section]
\newtheorem{corollary}{Corollary}[section]

\newtheorem{lemma}{Lemma}[section]
\newtheorem{remark}{Remark}[section]
\numberwithin{equation}{section}
\newcommand{\mc}{\mathcal}
\newcommand{\mb}{\mathbb}

\newcommand{\mf}{\mathfrak}

\newcommand{\R}{\mathbb{R}}
\newcommand{\wt}{\widetilde}

\newcommand{\be}{\beta}
\newcommand{\ze}{\zeta}

\newcommand{\va}{\varepsilon}
\newcommand{\al}{\alpha}
\newcommand{\pa}{\partial}
\newcommand{\Om}{\Omega}

\newcommand{\lm}{\lambda}

\newcommand{\sg}{\sigma}

\newcommand{\ga}{\gamma}
\newcommand{\de}{\delta}
\newcommand{\De}{\Delta}
\newcommand{\na}{\nabla}
\newcommand{\te}{\theta}
\newcommand{\var}{\vartheta}
\date{} 
\begin{document}

\title{Uniform Boundary Estimates in  Homogenization of Higher Order Elliptic Systems  }
\author{Weisheng Niu\thanks{ Corresponding author.
Supported in part by the NSF of China (11701002)
and Anhui Province (1708085MA02).} \quad \quad
 Yao Xu \thanks{Contributes equally as the first author} }
\maketitle
\pagestyle{plain}
\begin{abstract}
This paper focuses on the uniform boundary estimates in  homogenization of a family of higher order elliptic operators $\mc{L}_\va$, with rapidly oscillating periodic coefficients. We derive uniform boundary $C^{m-1,\lm} (0\!<\!\lm\!<\!1)$, $ W^{m,p}$ estimates in $C^1$ domains, as well as  uniform boundary $C^{m-1,1}$ estimate in $C^{1,\te} (0\!<\!\te\!<\!1)$ domains without the symmetry assumption on the operator. The proof, motivated by the profound work ``S.N. Armstrong and C.~K. Smart, Ann. Sci. \'Ec. Norm. Sup\'er. (2016),   Z. Shen, Anal. PDE (2017)'', is based on a suboptimal convergence rate in $H^{m-1}(\Om)$. Compared to ``C.E. Kenig, F. Lin  and Z. Shen, Arch. Ration. Mech. Anal. (2012), Z. Shen, Anal. PDE (2017)", the convergence rate obtained here does not require the symmetry assumption on the operator, nor additional assumptions on the regularity of $u_0$ (the solution to the homogenized problem), and thus might be of some independent interests even for second order elliptic systems.
\end{abstract}


\section{Introduction}
This paper is aimed to investigate the uniform boundary estimates in homogenization of the following $2m$-order elliptic system,
\begin{equation} \label{eq1}
 \begin{cases}
 \mc{L}_\varepsilon u_\varepsilon =f  &\text{ in } \Omega,  \vspace{0.1cm}\\
 Tr (D^\ga u_\varepsilon)=g_\ga  & \text{ on } \partial\Omega \quad\text{for  } 0\leq|\ga|\leq m-1,
\end{cases}
\end{equation}
where $\Om \subset \R^d, d\geq1,$ is a bounded Lipschitz domain, $$ (\mc{L}_\varepsilon u_\varepsilon )_i  = (-1)^{m}\sum_{|\alpha|=|\beta|=m} D^\alpha (A_{ij}^{\alpha \beta}(x/\varepsilon )D^\beta u_{\varepsilon j}), ~~   1\leq i, j\leq  n,$$
$ u_{\varepsilon j}$ denotes the $j$-th component of the $\mb{R}^n$-valued vector function $u_\varepsilon$, $\alpha, \beta,\ga$ are multi indexes with components $\alpha_k, \beta_k,\ga_k, k=1,2,...,d$, and
$$ |\alpha|=\sum_{k=1}^d \alpha_k, ~~D^\alpha=D_{x_1}^{\alpha_1} D_{x_2}^{\alpha_2}\cdot\cdot\cdot D_{x_d}^{\alpha_d}.  $$
The coefficients matrix $A(y)=(A_{ij}^{\alpha \beta}(y)), 1 \leq i, j \leq n,$  is real, bounded measurable, satisfying the strong ellipticity condition
\begin{align}\label{cod1}
 \mu |\xi|^2\leq \sum_{|\al|=|\be|=m} A_{ij}^{\al\be}(y)\xi^i_\al\xi^j_\be\leq \frac{1}{\mu} |\xi|^2  \quad\text{for }\, a.e.\,  y\in \R^{d},
\end{align}
where  $\mu>0, \xi=(\xi_\al)_{|\al|=m}, \xi_\al=(\xi^1_\al,...,\xi_\al^n)\in \R^n$, as well as the  periodicity condition
\begin{align}\label{cod3}
A(y+z)=A(y), \quad \text{for any } z\in \mathbb{Z}^{d} \text{ and }\, a.e.\, y\in \mathbb{R}^{d}.
\end{align}

The regularity estimate uniform in $\va>0$ is one of the main concerns in quantitative homogenization.  For second order elliptic operators, this issue has been studied extensively.  In the celebrated work of M. Avellaneda and F. Lin \cite{al87, al89, al91}, by using a compactness method, the
  interior and boundary H\"older estimate, $W^{1,p}$ estimate  and Lipschitz
  estimate were obtained for second order  elliptic systems with H\"{o}lder continuous coefficients and Dirichlet conditions in bounded $C^{1,\te}$ domains.
  The uniform boundary Lipschitz estimate for the Neumann problem has been a longstanding open problem, and was recently settled  by C. Kenig, F. Lin and Z. Shen in \cite{klsa1}.  Interested readers may refer to \cite{gsh, shenapde2015,kpar2015, shen2016approximate} and references therein for more applications of compactness method in quantitative homogenization.
  More recently, another fabulous scheme, which is based on convergence rates, was formulated in \cite{armstrongan2016} to investigate uniform  (interior) estimates in stochastic homogenization. This approach was further developed in \cite{armstrongcpam2016, shenan2017} for second order elliptic systems with periodic and almost periodic coefficients. Using this method, the large scale  interior or boundary  Lipschitz estimates for second order elliptic operators were studied  \cite{armstrongan2016, armstrongcpam2016, shenan2017}, see also \cite{Gloria2015, gloria2014regularity, armstrongar2016, zhuge2017, armstrong2017quantitative} for more related results.

Relatively speaking,  few quantitative results were known in the homogenization of higher order  elliptic equations previously,
 although results on qualitative homogenization have been obtained for many years \cite{lions1978}. Very recently, the optimal $O(\va)$ convergence rate in the $L^2(\R^d)$
for higher order  elliptic equations was obtained in \cite{ks,pastukhova2016, pastukhova2017}. In \cite{Suslina2017-D, Suslina2017-N}, some interesting two-parameter resolvent estimates were established in homogenization of general higher order elliptic systems with periodic coefficients in
bounded $C^{2m}$ domains.  Meanwhile, in \cite{nsx} we investigated the sharp $O(\va)$ convergence rate in Lipschitz domains. Under the assumptions that $A$ is symmetric and $u_0\!\in\!H^{m+1}(\Om)$,  the optimal $O(\va)$ convergence rate was obtained in $W^{m-1, q_0}(\Om), q_0=2d/(d-1)$. The uniform interior $W^{m,p}$ and  $C^{m-1,1}$ estimates were also established.

As a continuation of \cite{nsx}, in this paper we investigate the uniform boundary estimates in the homogenization of higher order elliptic systems.
Let $\psi: \R^{d-1}\rightarrow \R$ be a $C^1$ function with
\begin{align}\label{c1c}
\begin{split}
 &\psi(0)=0, \quad|\na\psi|\leq M,\\
 &\sup\Big\{ |\na \psi(x')-\na \psi(y') |: x',y'\in\R^{d-1} \text{ and } |x'-y'|\leq t \Big\}\leq \tau(t),
  \end{split}
\end{align}
where $\tau(t)\longrightarrow 0$ as $ t\longrightarrow 0^+.$
Set
\begin{align}\label{drder}
\begin{split}
&D_r=D(r,\psi)=\Big\{ (x',x_d)\in \mb{R}^d: |x'|<r \text{ and } \psi(x')<x_d<\psi(x')+r\Big\},\\
&\De_r=\De(r,\psi)=\Big\{ (x',\psi(x'))\in \mb{R}^d: |x'|<r \Big\}.
\end{split}
\end{align}
The main results of this paper can be stated as follows.
\begin{theorem}\label{thholder}
Suppose that the coefficient matrix $A=A(y)$ satisfies the conditions  (\ref{cod1})--(\ref{cod3}) and $u_\va \in H^{m}(D_1; \R^n)$ is a weak solution to
 \begin{equation*}
\begin{cases}
 \mc{L}_\va u_\va=F     &\text{ in } D_1,   \\
 Tr (D^\ga u_\va)=D^\ga G  &\text{ on }  \De_1 \quad\text{for  } 0\leq|\ga|\leq m-1,
 \end{cases}\end{equation*}
where $G\in C^{m-1,1}(D_1; \R^n),  F\in L^p(D_1; \R^n)$ with $p>\max\big\{d/(m+1),  2d/(d+2m-2), 1\big\}$. Then for any $0<\lm <\min\{m+1-d/p, 1\}$ and any  $\va\leq r < 1,$
\begin{align}\label{tholdere1}
\Big(\fint_{D_r}|\na^m u_\va|^2\Big)^{1/2}\leq C r^{\lm-1}\bigg\{&\Big(\fint_{D_1}|u_\va|^2 \Big)^{1/2}+\Big(\fint_{D_1} |F|^p\Big)^{1/p} + \| G\|_{C^{m-1,1}(D_1)} \bigg\},
\end{align}
where $C$ depends only on $d,n,m, \lm,\mu,p$ and $\tau(t)$ in (\ref{c1c}).
\end{theorem}
Estimate (\ref{tholdere1}) can be viewed as the $C^{m-1,\lm}$ estimate uniform down to the scale $\va$ in $C^1$ domains for higher order  elliptic operators $\mc{L}_\va$.  In addition to the assumptions in Theorem \ref{thholder}, if $A\in V\!M\!O (\mathbb{R}^d),$ i.e., \begin{align}\label{vmo}
  \sup_{ x\in \mathbb{R}^d,\,0<r<t } \fint_{B(x,r)}|A(y)-\fint_{B(x,r)}A|dy\leq \varrho(t), ~~~ 0<t\leq 1,
\end{align}for some nondecreasing continuous function $\varrho(t)$ on $[0,1]$  with $ \varrho(0)= 0$.
Then a standard blow-up argument  gives the following full-scale boundary $C^{m-1,\lm}$ estimate
\begin{align} \label{holdere2}
\|u_\va\|_{C^{m-1,\lm}(D_{1/4})}\leq C \bigg\{\Big(\fint_{D_1}|u_\va|^2 \Big)^{1/2}& +\Big(\fint_{D_1} |F|^p\Big)^{1/p} + \| G\|_{C^{m-1,1}(D_1)}\bigg\}.
\end{align}

 We also mention that the restriction $p>\max\{ d/(m+1), 1\}$ is made to ensure $C^{m-1,\lm}$ estimate of the solution $u_0$  to the homogenized system, which plays an essential role in the proof of the theorem. The restriction $ p>2d/(d+2m-2)$ is used to ensure that $F\in H^{-m+1}(\Om)$, since our proof is based on the convergence result in Theorem 1.4 (see Lemma \ref{lemm3.1} for details).  Although the assumption on the regularity of $F$ in Theorem 1.1 is not sharp, see Corollary \ref{co4.1}  for the full scale uniform $C^{m-1,\lm}$ estimate of $u_\va$, it is enough for us to derive the following uniform $W^{m,p}$ estimate on $u_\va$.

\begin{theorem}\label{twmp}
Let $\Omega$ be a bounded $C^1$ domain in $\mb{R}^d$. Suppose that the coefficient matrix $ A\in V\!M\!O(\R^d)$ satisfies (\ref{cod1})--(\ref{cod3}) and $u_\varepsilon\in H^m(\Om; \R^n)$ is a weak solution to
\begin{equation*}
\begin{cases}
 \mc{L}_\va u_\va=\sum_{|\al|\leq m}D^\al f^\al &\text{ in } \Om,\\
 Tr(D^\ga u_\va)=g_\ga&\text{ on } \pa\Om \quad\text{for}~ 0\leq |\ga|\leq m-1,
 \end{cases}\end{equation*}
where $\dot{g} =\{g_\ga\}_{|\ga|\leq m-1}\in \dot{B}^{m-1/p}_p(\pa\Om; \R^n)$ and $ f^\al\in L^p(\Om;\R^n)$ for $|\al|\leq m, 2\leq p< \infty$.
Then
 \begin{align}\label{twm11}
 \| u_\varepsilon\|_{W^{m,p}(\Om)}   \leq C_p \,\bigg\{ \sum_{|\al|\leq m}\|f^\al\|_{L^p(\Omega)} + \|\dot{g}\|_{\dot{B}^{m-1/p}_p(\pa\Om) }\bigg\},
\end{align}
where the constant $C_p$ depends only on $p,d,n,m,\mu,\Omega$ and $\varrho(t)$ in (\ref{vmo}).
\end{theorem}
    We refer readers to Section 2 for the definition of the Whitney-Besov space $\dot{B}^{s}_p(\partial\Omega; \R^n)$. Note that although the result presented in Theorem \ref{twmp} focuses on the case $p\geq2$, by a standard duality argument, it still holds for $1\!<\!p\!<\!2$. We also mention that the uniform $W^{1,p}$ estimates in the homogenization of second order elliptic systems have been studied largely, see e.g., \cite{gs, gs2, armstrongjfa2016, xujde2017}.  Theorem \ref{twmp}  generalizes the uniform $W^{1,p}$ estimates for second order  elliptic systems to higher order  elliptic systems.

Our third result gives the uniform boundary $C^{m-1,1}$ estimate of $u_\va$ in $C^{1,\te}\, (0<\te<1)$ domains.   Let $D_r, \De_r$ be defined as in (\ref{drder}), and let the defining function $\psi\in C^{1,\te} (\R^{d-1})$ with\begin{align}\label{c2c}
\psi(0)=0, \quad\|\na\psi\|_{C^\te(\R^{d-1})}\leq M_1.
\end{align}
\begin{theorem}\label{tlip}
Assume that $A$ satisfies  (\ref{cod1})--(\ref{cod3}). Let $u_\va\in H^m(D_{1}; \mb{R}^n )$ be a weak solution to
\begin{equation*}
\begin{cases}
 \mc{L}_\va u_\va= \sum_{|\al|\leq m-1}D^\al f^\al    &\text{ in } D_1,   \\
 Tr (D^\ga u_\va)=D^\ga G  &\text{ on }  \De_1 \quad \text{for  } 0\leq|\ga|\leq m-1,
 \end{cases}\end{equation*}
where $f^\al\in L^{q}(D_1; \R^n) $ with $q>d, q\geq 2$, and $ G\in C^{m,\sg}(D_1; \R^n)$ for some $0<\sg\leq \te$.
Then for any $\va\leq r<1$, we have
\begin{align}\label{tlipre1}
\Big(\fint_{D_r}|\na^m u_\va|^2\Big)^{1/2}\leq C  \bigg\{\Big(\fint_{D_1}|u_\va|^2 \Big)^{1/2}
+ \sum_{|\al|\leq m-1}\Big(\fint_{D_1}|f^\al|^q\Big)^{1/q} + \|G \|_{C^{m,\sg}(D_1)}\bigg\},
\end{align}
where $C$ depends only on $d,n,m,\mu,  q, \sigma, \te$ and $M_1$.
\end{theorem}

Similar to \ref{tholdere1}, estimate (\ref{tlipre1}) is the $C^{m-1,1}$ estimate uniform down to the scale $\va$ for the operator $\mc{L}_\va$, which separates the large-scale estimates due to the homogenization process from the small-scale estimates related to the smoothness of the coefficients.  If in addition, $A$ is H\"{o}lder continuous,
 i.e., there exist $  \Lambda_0>0,  \tau_0\in(0,1) $ such that
\begin{align}\label{hol}
|A(x)-A(y)|\leq \Lambda_0 |x-y|^{\tau_0} \quad \text{ for any } x, y\in \mathbb{R}^d,
\end{align}
we can derive the full-scale boundary $C^{m-1,1}$ estimate
\begin{align}\label{co511}
\|\na^mu_\va\|_{L^{\infty}(D_{1/4})}\leq C  \bigg\{\Big(\fint_{D_1}|u_\va|^2 \Big)^{1/2}    +\sum_{|\al|\leq m-1}\Big(\fint_{D_1}|f^\al|^q\Big)^{1/q} + \|G \|_{C^{m,\sg}(D_1)}\bigg\}.
\end{align}
This generalises the boundary Lipschitz estimates in \cite{al87,shenan2017} for second order  elliptic systems to higher order  elliptic systems.

 Note that Theorem \ref{tlip} does not require the symmetry assumption on the coefficient matrix $A$. Therefore, it may be of some independent interests even for second order  elliptic systems.
  Recall that the symmetry assumption on the coefficient matrix $A$ is made in \cite{klsa1} to establish the uniform boundary Lipschitz estimate for second order elliptic systems with Neumann boundary conditions. Such  an assumption was removed in \cite{armstrongcpam2016}, where the boundary Lipschitz estimate was obtained for both the Dirichlet and Neumann problems (of second order) with almost periodic coefficients. When $m=1$, without essential difficulties we  may extended the uniform Lipschitz estimate in Theorem \ref{tlip} to Neumann boundary problems. However, our investigations do not rely on the nontangential maximum function estimates, which had played an essential role in \cite[p.1896]{armstrongcpam2016}. This may allow one to treat more general elliptic systems.

Finally, we mention that the requirements on smoothness of coefficients and the domain for uniform estimates in Theorem \ref{thholder} to Theorem \ref{tlip} are the same as those for second order  elliptic systems \cite{shenan2017}. Therefore, results in theorems above, combined with the interior estimates in our previous paper \cite{nsx},  present a unified description on the uniform regularity estimates in homogenization of $2m$-order elliptic systems in the divergence form. The counterpart for higher order  elliptic operators of non-divergence form will be presented in a separate paper shortly.

 The proofs of theorems above rely on the following convergence result.
 \begin{theorem}\label{thcon}
Suppose that $\Omega$ is a bounded Lipschitz domain in $\mathbb{R}^d, d\geq1$, and
 the coefficient matrix $A$ satisfies (\ref{cod1})--(\ref{cod3}). Let $u_\varepsilon, u_0$ be the weak solutions
 to the Dirichlet problem (\ref{eq1}) and the homogenized problem (\ref{hoeq1}), respectively.
Then for $0<\va<1$ and any $0<\nu<1 $, we have\begin{align}\label{thconre1}
&\|u_\varepsilon-u_0\|_{H^{m-1}_0(\Omega)}\leq C_\nu\varepsilon^{1-\nu} \left\{\|\dot{g}\|_{W\!A^{m,2}(\partial\Omega)}+\|f\|_{H^{-m+1}(\Omega)}\right\},
\end{align} where $C_\nu$ depends only on $ d, n, m,\nu,\mu$ and $\Om.$
If in addition $ A$ is symmetric, i.e. $A=A^*$, then
\begin{align}\label{thconre2}
 \| u_\varepsilon-u_0\|_{H^{m-1}_0(\Omega)} \leq C \varepsilon  \ln (1/\va)  \left\{\|\dot{g}\|_{W\!A^{m,2}(\partial\Omega)}+\|f\|_{H^{-m+1}(\Omega)}\right\},
\end{align}
where $C$ depends only on $d, n, m,\mu$ and $\Om.$
\end{theorem}
The error estimates above can be viewed as a counterpart in general Lipschitz domains for the convergence rates obtained in \cite{ks,pastukhova2016, pastukhova2017, Suslina2017-D}. Estimate (\ref{thconre1}) is new and may be of some independent interests even for second order elliptic systems. Recall that sharp convergence rate has been extensively studied for second order elliptic equations. The estimate
\begin{align*}
\|u_\varepsilon-u_0\|_{L^{2} (\Omega)} \leq C \varepsilon \|u_0\|_{H^{2}(\Omega)}
\end{align*}
has been obtained for second order elliptic equations in divergence form in $C^{1,1}$ domains \cite{grisoas2004,suslinaD2013, suslinaN2013}, as well as in Lipschitz domains with additional assumptions $u_0\in H^2(\Om)$ and $A=A^*$ \cite{shenan2017, nsx}. In \cite{klsal2,xusiam2016}, the $O[\va \ln(1/\va)]$ convergence rate like  (\ref{thconre2}) was obtained for second order elliptic systems under the assumption that $A=A^*$.
Compared with the reference aforementioned, our estimate (\ref{thconre1}), although suboptimal, holds in general Lipschitz domains and needs neither the symmetry of $A$, nor additional regularity assumptions of $u_0$.  Moreover, the assumptions on the regularity of $A, \dot{g}, f$ are also rather general. To the best of the authors' knowledge, optimal or suboptimal convergence rate under such weak conditions seems to be unknown previously even for second order elliptic systems.

The proof of Theorem \ref{thcon} follows the line of \cite{klsal2,suslinaD2013}. The first step is to derive an estimate like
\begin{align*}
  \|u_\varepsilon-u_0\|_{L^{2} (\Omega)} \leq C \varepsilon^{(1/2)^-}  \left\{\|\dot{g}\|_{W\!A^{m,2}(\partial\Omega)}+\|f\|_{H^{-m+1}(\Omega)}\right\}.
  \end{align*}
When $A$ is symmetric, this was done with the help of the nontangential maximum function estimate, which gives proper controls on $u_0$ near the boundary $\pa\Om$, see \cite{klsal2,xusiam2016,shenan2017} for the details.
Unfortunately, if $A$ is not symmetric and the domain is just Lipschitz (or even $C^1$) the nontangential maximum function estimate is not in hand.
Instead, we will take  advantage of some weighted estimate of $u_0$ (see Lemma \ref{l3.2} ) to achieve the goal. With these estimates at our disposal, we then modify the duality argument in \cite{suslinaD2013} (see also \cite{suslinaN2013, shenan2017}) with proper weight to derive the desired convergence rate.


Armed with Theorem \ref{thcon}, our proof of Theorems \ref{thholder} and \ref{tlip} follows the scheme in \cite{armstrongan2016, shenan2017}, which  roughly speaking  is a three-step argument:
\begin{enumerate}
\item[(i)] Establish the convergence rate in $L^2(\Om)$ in terms of boundary data $g$ and the forcing term $f$, i.e., the error estimate  like $$\|u_\va-u_0\|_{L^2(\Om)}\leq C \va^{\sigma_0}\big\{ \text{   norms of data } g\text{ and }  f \big\},  \text{ for some }  0<\sigma_0\leq 1;$$
\item[(ii)]  Prove that $u_\va$  satisfies the  flatness property, i.e., how well it could be approximated by a affine functions as $u_0$ does;
\item[(iii)]  Iterate step (ii) down to the scale $\va$, with the help of the error estimate in the first step.
\end{enumerate}
Note that  (\ref{thconre1}) gives (i), we can thus pass to Step (ii). We shall adapt some ideas in \cite{armstrongan2016, shenan2017} to verify that $u_\va$ satisfies  the so-called flatness property. However, instead of estimating how well  $u_\varepsilon$ is approximated by ``affine" functions as in \cite{armstrongan2016,shenan2017}, we estimate how well $u_\varepsilon$ is approximated by  polynomials of  degree $m-1$ and $m$, respectively. By a proper iteration argument, we then derive the desired large-scale $C^{m-1,\lm} (0\!<\!\lm\!<1)$ and $C^{m-1,1}$ estimates.
 The corresponding full scale estimates (\ref{co311}) and (\ref{co511}) follow from a standard blow-up argument.

Finally, the proof of Theorem \ref{twmp} relies on the boundary H\"{o}lder estimate (\ref{tholdere1}) and a real variable argument originated from \cite{caffarelli1998} and further developed in \cite{shenan2005, shenad2007}. The key idea is to reduce the $W^{m,p}$ estimate (\ref{twm11})
 to a reverse H\"{o}lder inequality of the corresponding homogeneous problem, see Lemma \ref{lemm4.1} for the details.


\section{Preliminaries}
\subsection{Function spaces}

To begin with, let us give the definitions of some function spaces involved next. Let $\Om $ be a bounded Lipschitz domain in  $\R^d.$
Let $H^m(\Omega; \mathbb{R}^n )$ and $H_0^m(\Omega; \mathbb{R}^n )$ with dual $H^{-m}(\Omega; \mathbb{R}^n )$,  be the conventional Sobolev spaces of $\mathbb{R}^n$-valued  functions.
For $0<s<1, 1<p<\infty$ and any nonnegative integer $k$, let $B^{k+s}_p(\Om)$ be the Besov space with norm, see \cite{grisvard} (p.17)
$$ \|u\|_{B^{k+s}_p(\Om)}= \sum_{0\leq\ell\leq k}\|\na^\ell u\|_{L^{p}(\Om)}+\sum_{|\ze|=k}\bigg\{\int_{\Om} \int_{\Om} \frac{|D^\ze f(x)-D^\ze f(y)|^p}{|x-y|^{d +sp}}\,dx\,dy\bigg\}^{1/p}.$$
Since $\Om$ is a bounded Lipschitz domain, $B^{k+s}_p(\Om)$ consists of the restrictions to $\Om$ of functions in $B^{k+s}_p(\R^d)$ \cite[p.25]{grisvard}.

Also define the  Whitney-Besov space $\dot{B}^{m-1+s}_p(\partial\Omega; \R^n)$ as
the closure of the  set of arrays
$$\left\{  \{ D^\alpha \mathcal{U} \}_{|\alpha|\leq m-1}:  \mathcal{U}\in C_c^\infty(\mathbb{R}^d) \right\},$$
under the norm  $$\| \dot{u} \|_{\dot{B}^{m\!-1\!+s}_p(\partial\Omega)}  =\sum_{|\alpha|\leq m-1}\Big\{\|u_\alpha\|_{L^p(\partial \Omega)}+\Big(\int_{\pa\Om} \int_{\pa\Om} \frac{|u_\al(x)-u_\al(y)|^p}{|x-y|^{d-1+sp}}dS_x dS_y\Big)^{1/p}\Big\},$$
where $ \dot{u}= \{u_\alpha\}_{|\alpha|\leq m-1},$ see e.g., \cite{agran2007}.

Define the Whitney-Sobolev space $W\!A^{m,p}(\partial\Omega, \mathbb{R}^n)$ as the completion of the set of arrays of $\mathbb{R}^n$-valued functions
$$\left\{\{ D^\alpha \mathcal{G}\mid_{\partial\Omega}\}_{|\alpha|\leq m-1}:  \mathcal{G}\in C_c^\infty(\mathbb{R}^d; \mathbb{R}^n ) \right\},
$$
under the norm
$$ \| \dot{g} \|_{W\!A^{m,p}(\partial\Omega)} =\sum_{|\alpha|\leq m-1} \|g_\alpha\|_{L^p(\partial \Omega)} + \sum_{|\alpha|=m-1} \|\nabla_{tan}  g_\alpha\|_{L^p(\partial \Omega)}.
$$  for any $ \dot{g}= \{g_\alpha\}_{|\alpha|\leq m-1} \in W\!A^{m,p}(\partial\Omega, \mathbb{R}^n)$\cite{mms}.

\subsection{Qualitative Homogenization}
Under the ellipticity condition
(\ref{cod1}),
for any $\dot{g}\in W\!A^{m,2}(\partial\Omega, \mathbb{R}^n)$ and $ f\in H^{-m}(\Omega;\mathbb{R}^n)$,
  Dirichlet problem (\ref{eq1}) admits a unique weak solution $u_\varepsilon$ in $ H^m(\Omega;\mathbb{R}^n)$ such that
 $$
 \|u_\varepsilon\|_{H^m(\Omega)} \leq C \left\{\|f\|_{H^{-m}(\Omega)} +\|\dot{g}\|_{W\!A^{m,2}(\partial\Omega)}\right\},
 $$
 where $C$ depends only on $d$, $m$, $n$, $\mu$ and $\Omega$. It is known that (see e.g., \cite{lions1978, pastukhova2016})
under the additional periodicity condition (\ref{cod3}),
the operator $\mathcal{L}_\varepsilon$ is G-convergent to  $\mathcal{L}_0 $, where
 \begin{align*}
(\mathcal{L}_0 u )_i = \sum_{|\alpha|=|\beta|=m} (-1)^m D^\alpha (\bar{A}_{ij}^{\alpha \beta}D^\beta u_j)
\end{align*}
is an elliptic operator of order $2m$ with constant coefficients,
 $$\bar{A}_{ij}^{\alpha \beta}=\sum_{|\gamma|=m}\frac{1}{|Q|}\int_Q \Big[A_{ij}^{\alpha \beta}(y)
 + A_{i\ell}^{\alpha \gamma}(y)D^\gamma \chi_{\ell j}^\beta (y)\Big]dy.
 $$
Here $Q=[0,1)^d$, $\chi=(\chi_j^\ga)= (\chi^\gamma_{ij})$ is the matrix of correctors for the operator $\mathcal{L}_\varepsilon$
given by
\begin{equation}\label{corrector}
\begin{cases}
 \sum_{|\alpha|=|\beta|=m}  D^\alpha \big\{A_{ik}^{\alpha \beta}(y)D^\beta \chi_{kj}^\ga (y)\big\}=- \sum_{|\alpha| =m}  D^\alpha   A_{ij}^{\alpha \ga}(y)  ~~    \text{ in } \R^d,\vspace{0.3cm}\\
 \chi_j^\ga (y) \quad\text{ is 1-periodic } \quad\text{ and }\quad
 \int_Q \chi_j^\ga (y)=0.
 \end{cases}
\end{equation}
The matrix $(\bar{A}_{ij}^{\alpha \beta})$ is bounded and
satisfies the coercivity condition (\ref{cod1}). Thus the following homogenized problem of (\ref{eq1}),
\begin{equation} \label{hoeq1}
 \begin{cases}
 \mathcal{L}_0  u_0 =f &\text{ in } \Omega,  \vspace{0.3cm}\\
 Tr (D^\gamma u_0)=g_\gamma & \text{ on } \partial\Omega\quad \text{for }  0\leq|\gamma|\leq m-1,
\end{cases}
\end{equation}
admits a unique weak solution $u_0\in H^m(\Omega; \mathbb{R}^n )$,  satisfying
$$ \|u_0\|_{H^m(\Omega)} \leq C \left\{\|f\|_{H^{-m}(\Omega)} +\|\dot{g}\|_{W\!A^{m,2}(\partial\Omega)}\right\}.$$

For $1\leq i,j\leq n$ and multi indexes $\al,\be$ with $|\alpha|=|\beta|=m,$ set
\begin{align}\label{duc}
B_{ij}^{\alpha \beta}(y)=A_{ij}^{\alpha \beta}(y)+ \sum_{|\ga|=m}A_{ik}^{\alpha \ga}(y)  D^\ga \chi_{kj}^\beta (y)-\bar{A}_{ij}^{\alpha \beta}.
\end{align}
By the definitions of  $\chi^\ga(y)$ and $\bar{A},$
  for any $1\leq i,j\leq n$ and any multi indexes $\al,\be$ with $|\alpha|=|\beta|=m,$  $B_{ij}^{\alpha \beta}(y)\in L^2(Q) $ is 1-periodic with zero mean, and
$\sum_{|\alpha|=m}D^\alpha B_{ij}^{\alpha \beta}(y)=0.
$
Therefore,  there exists a function $ \mf{B}_{ij}^{\ga\alpha\beta} $ such that
 \begin{align*}
 \mf{B}_{ij}^{\ga\alpha\beta}  =- \mf{B}_{ij}^{\alpha\ga\beta} , ~~\sum_{|\ga|=m}D^\ga \mf{B}_{ij}^{\ga\alpha\beta} = B_{ij}^{\alpha \beta}
\quad\text{and} \quad \|\mf{B}_{ij}^{\ga\alpha\beta} \|_{H^m(Q)}\leq  C  \|B_{ij}^{\alpha\beta}\|_{L^2(Q)},
 \end{align*}
  where $C$ depends only on $d, n, m$, see \cite[Lemma 2.1]{nsx}.

Let $\mathcal{L}^*_\varepsilon$ be the adjoint operators of $\mathcal{L}_\varepsilon$,  i.e.,
\begin{align}\label{adj}
\mathcal{L}^*_\varepsilon=  (-1)^{m}\sum_{|\alpha|=|\beta|=m} D^\alpha\left(A^{*\alpha\beta}
(x/\va)D^\beta  \right), \quad ~A^*=(A_{ij}^{*\alpha\beta})= ( A_{ji}^{\beta \alpha}).
\end{align}
Parallel to (\ref{corrector}), we can introduce the matrix of correctors $\chi^*= (\chi^{*\ga}_{j})= (\chi^{*\ga}_{ij})$ for  $\mathcal{L}^*_\varepsilon$,
 \begin{equation}\label{dualcorrector}
\begin{cases}
 \sum_{|\alpha|=|\beta|=m}  D^\alpha \big\{A_{ik}^{*\alpha \beta}(y)D^\beta \chi_{kj}^{*\ga}(y)\big\}=- \sum_{|\alpha| =m}  D^\alpha   A_{ij}^{*\alpha \ga}(y)  ~~    \text{ in } \R^d,\vspace{0.3cm}\\
 \chi_j^{*\ga} (y) \quad\text{ is 1-periodic } \quad\text{ and }\quad
 \int_Q \chi_j^{*\ga} (y)=0,
 \end{cases}
\end{equation}
We can also introduce the dual correctors $ \mathfrak{B}^{*\gamma \alpha \beta}(y)$ of $\chi^*$.
It is not difficult to see that $ \chi^{*\ga}$ and $ \mathfrak{B}^{*\gamma \alpha \beta}$ satisfy the same properties as $\chi^\ga$ and $ \mathfrak{B}^{\gamma \alpha \beta}$, since $A^*$ satisfies the same conditions as $A$.

\subsection{ Smoothing operators and auxiliary estimates}

For any fixed $\varphi\in C_c^\infty(B(0,\frac{1}{2}))$ such that $\varphi>0$ and $\int_{\mb{R}^d} \varphi (x)dx=1$, set  $\varphi_\varepsilon=\frac{1}{\varepsilon^d} \varphi(\frac{x}{\varepsilon}) $ and define
\begin{align*}
S_\varepsilon(f)(x)=\int_{\mb{R}^d} \varphi_\varepsilon(x-y)f(y)\, dy, \quad\text{and}\quad  S^2_\varepsilon= S_\varepsilon \circ S_\varepsilon.
\end{align*}
Denote $\de(x)=dist(x, \pa\Om)$, $\Omega^{ \varepsilon}=\{x\in\Omega: \de(x)> \varepsilon\},~~
\Omega_{ \varepsilon}=\{x\in\Omega: \de(x)< \varepsilon\}.$
\begin{lemma}\label{l2.2}
Assume that $f\in L^p(\R^d)$ for some $1\leq p<\infty$ and $g\in L_{loc}^p(\mb{R}^d)$ is 1-periodic. Let $h\in L^\infty(\mb{R}^d)$ with compact support $\Om^{3\va}$. Then
 \begin{align}
&\|g(x/\varepsilon)S_\varepsilon(f)(x) h(x) \|_{L^p(\Om^{3\va};\, \de )}\leq C   \|g\|_{L^p(Q)} \|f\|_{L^p(\Om^{2\va};\, \de )},\label{l2.2re2}\\
&\|g(x/\varepsilon)S_\varepsilon(f)(x) h(x) \|_{L^p(\Om^{3\va};\, \de^{-1})}\leq C   \|g\|_{L^p(Q)} \|f\|_{L^p(\Om^{2\va};\, \de^{-1} )}. \label{l2.2re3}
\end{align}
where  $\|u\|_{L^p(\Om;\, \de )}$ (similar for $\|u\|_{L^p(\Om;\, \de^{-1} )}$ ) denotes the weighted norm
$$ \|u\|_{L^p(\Om;\, \de)}=\Big(\int_\Om |u(x)|^p \de(x) \, dx\Big)^{1/p}.$$
\end{lemma}
\begin{proof} Observe that
\begin{align}\label{pl2201}
&\int_{\R^d} \big|g(x/\varepsilon)h(x) \int_{\mb{R}^d} \varphi_\varepsilon(x-y)f(y)\, dy \big|^p\de(x)\, dx \nonumber\\
&\leq C\int_{\Om^{3\va}}  |g(x/\varepsilon) |^p  \int_{\Om^{2\va}}  \varphi_\varepsilon(x-y)|f(y)|^p \de(y)\, dy  \Big\{\int_{\Om^{2\va}}\varphi_\varepsilon(x-y) \de(y)^{-q/p} dy \Big\}^{p/q}\de(x)\, dx \nonumber\\
&\leq C \int_{\Om^{2\va}} \int_{\Om^{3\va}}   |g(x/\varepsilon)|^p    \varphi_\varepsilon(x-y)dx\,|f(y)|^p \de(y)\, dy \nonumber\\
&\leq C \int_{Q} |g(z)|^pdz \int_{\Om^{2\va}} |f(y)|^p \de(y)\, dy,
\end{align}
where we have used Fubini's theorem  and the observation $$\int_{\Om^{2\va}}\varphi_\varepsilon(x-y) [\de(y)]^{-q/p}\,  dy \leq C [\de(x)]^{-q/p}$$ for the second inequality.  This gives (\ref{l2.2re2}). The proof of (\ref{l2.2re3}) is the same. \end{proof}

\begin{lemma}\label{l2.3}
Let $  \widetilde{\Omega}_\varepsilon= \{x\in \mb{R}^d: \de(x)<\varepsilon\}, f\in H^\ell(\R^d), \ell\geq 0$. Then for any multi index $\alpha, |\alpha|=\ell$,
\begin{align}
&\|S_\varepsilon(D^\alpha f)\|_{L^p(\Omega_{\varepsilon})}\leq C \varepsilon^{-\ell}\|f\|_{L^p( \widetilde{\Omega}_{2\varepsilon})},\label{l2.3re1}\\
&\|S_\varepsilon(D^\alpha f)\|_{L^p(\Omega^{3\varepsilon};\, \de )}\leq C \varepsilon^{-\ell} \|f\|_{L^p(\Omega^{\varepsilon};\,\de)}.\label{l2.3re3}
\end{align}
\end{lemma}
\begin{proof}
Inequality (\ref{l2.3re1}) was proved in \cite[Lemma 2.3]{nsx}, and the proof of (\ref{l2.3re3}) is quite similar. We provide it just for completeness.
\begin{align*}
&\|S_\varepsilon(D^\alpha f  )\|^p_{L^p(\Omega^{3\varepsilon};\,  \de )}=\int_{\Omega^{3\varepsilon}}  \Big|\int_{\mathbb{R}^d}
D^\alpha \varphi_\varepsilon (x-y)  f(y)  \, dy \Big|^p  \de(x) \, dx\nonumber\\
&\leq
\int_{\Omega^{3\varepsilon}} \int_{\Omega^{2\varepsilon}}
|D^\alpha\varphi_\varepsilon (x-y)|\, |f(y)|^p  \de(y)\, dy  \Big\{\int_{\Omega^{2\va}}
|D^\alpha\varphi_\varepsilon (x-y)|\, [\de(y)]^{-q/p}\, dy \Big\}^{p/q} \, \de(x)\, dx\nonumber\\
&\leq \frac{C}{\varepsilon^{p\ell}}
\int_{\Omega^{2\va}} |f(y) |^p  \de(y)\, dy,
\end{align*}
where  we have used Fubini's theorem and the observation
\begin{align*}
\int_{\Omega^{2\varepsilon}}
|D^\alpha\varphi_\varepsilon (x-y)|\, [\de(y)]^{-q/p} \, dy\leq C\int_{\Omega^{2\varepsilon}}
|D^\alpha\varphi_\varepsilon (x-y)|\, [\de(x)]^{-q/p} \, dy \leq C  \va^{-\ell} [\de(x)]^{-q/p}
\end{align*}for the last step.
\end{proof}

\begin{lemma}\label{l2.4}
Suppose that $f\in W^{1,q}(\R^d)$ for some $1< q<\infty$. Let $ \na^s f=(D^\al f)_{|\al|=s}.$ Then
\begin{align}\label{l2.4re1}
\|  S_\varepsilon(f)-f\|_{L^q(\Om^{2\va};\, \de)}\leq C \varepsilon   \|\nabla f\|_{L^q(\Om^\va;\, \de)}.
\end{align}
\end{lemma}
\begin{proof}
See \cite[Lemma 3.3]{xusiam2016} and also \cite[Lemma 2.2]{shenan2017} for the case $q=2$.
\end{proof}


\begin{lemma}\label{l2.6}
Assume that $A$ satisfies (\ref{cod1})--(\ref{cod3}), and $u_\va\in H^m(B(x_0,R)\cap\Om; \mb{R}^n )$ is a solution to $\mc{L}_\va u=  \sum_{|\alpha|\leq m}D^\alpha f^\alpha$ in $B(x_0,R)\cap\Om$ with $Tr (D^\ga u_\varepsilon)=D^\ga G$ on $B(x_0,R)\cap \pa\Om$ for some $G\in H^{m}(B(x_0,R)\cap\Om; \mb{R}^n )$ where $x_0\in\pa\Om$. Let $ f^\alpha \in L^2(B(x_0,R)\cap\Om; \mb{R}^n )$ for $ |\alpha|\leq m$. Then for $0\leq j\leq m$ and $0<r<R$,  we have
\begin{align}\label{cac2}
\int_{B(x_0,r)\cap\Om} |\nabla^j (u_\va-G)|^2&\leq \frac{C}{(R-r)^{2j}} \int_{B(x_0,R)\cap\Om}(|u_\va|^2+|G|^2)+ C R^{2m-2j} \int_{B(x_0,R)\cap\Om}|\na^m G|^2 \nonumber\\&\quad +C\sum_{|\al|\leq m} R^{4m-2j-2|\al|} \int_{B(x_0,R)\cap\Om}  |f^\al|^2,\end{align}
where $C$  depends only on $d,n, m,\mu$ and $\Om$.
\end{lemma}
\begin{proof}
It is obvious that
$v_\va=u_\va-G$ is a solution to
\begin{align*}
\begin{split}
&\mc{L}_\va v_\va=  \sum_{|\alpha|\leq m}D^\alpha f^\alpha + \sum_{|\alpha|=|\be|=m}D^\alpha\{A^{\al\be} D^\be G\} \quad\text{ in } B(x_0,R)\cap\Om,\\
 &Tr (D^\ga v_\va )=0   \quad\text{ on } B(x_0,R)\cap \pa\Om \quad\text{for } 0\leq |\ga|\leq m-1 .
 \end{split}
 \end{align*}
Let $\phi\in C_c^\infty (B(x_0,R))$ with $\phi=1$ in $B(x_0,r)$ and $|\na^k \phi|\leq C (R-r)^{-k}$. Multiplying $v_\va \phi^2$ and using   integration by parts,
we obtain that
\begin{align}\label{plcac1}
 \int_{B(x_0,R)\cap\Om} |\na^mv_\va|^2\phi^2
&\leq\sum_{|\alpha|\leq m} \Big\{C(\epsilon_0)R^{2m-2|\al|} \int_{B(x_0,R)\cap\Om} |f^\al|^2  +   \frac{\epsilon_0}{R^{2m-2|\al|}} \int_{B(x_0,R)\cap\Om} |D^\al (v_\va\phi^2)|^2 \Big\} \nonumber\\
&\quad+C(\epsilon_0) \sum_{|\alpha|=m} \int_{B(x_0,R)\cap\Om}  |D^\al G|^2 + \epsilon_0\int_{B(x_0,R)\cap\Om} |\nabla^m v_\va|^2\phi^2 \nonumber \\
&\quad+ \sum_{j=0}^{m-1}\frac{C\epsilon_0+C}{(R-r)^{2m-2j}}\int_{(B(x_0,R)\setminus B(x_0,r))\cap\Om} |\nabla^jv_\va|^2.
\end{align}
Note that $v_\va\phi^2 \in H^m_0( B(x_0,R)\cap\Om)$. Using Poincar\'{e}'s inequality  
and setting $\epsilon_0$ small enough, we may obtain from (\ref{plcac1}) that
\begin{align}\label{plcac2}
\int_{B(x_0,r)\cap\Om} |\na^m(u_\va-G)|^2&\leq \sum_{j=0}^{m-1}\frac{C}{(R-r)^{2m-2j}} \int_{(B(x_0,R)\setminus B(x_0,r))\cap\Om}|\na^j (u_\va-G)|^2 \nonumber\\ &+ C\Big\{\sum_{|\al|\leq m} R^{2m-2|\al|}\int_{B(x_0,R)\cap\Om} |f^\alpha|^2+ \sum_{|\al|= m} \int_{B(x_0,R)\cap\Om}|D^\al G|^2 \Big\},
\end{align}where $C$ depends only on $d,n, m$ and $\mu$, but never on $\va, R.$
 The estimate (\ref{cac2}) follows from (\ref{plcac2}) in the same way as Corollary 23 in \cite{barton2016} by an induction argument. \end{proof}

\begin{remark}\label{remark0}
It is possible to replace the $L^2$ norm of $f^\al$  with $|\al|<m$  in (\ref{cac2}) by the $L^p$ norm for some $1<p<2$. For example, assume that $f^\al=0$ for $1\leq |\al|\leq m$. We may prove that
\begin{align*}
\int_{B(x_0,r)\cap\Om} |\nabla^j (u_\va-G)|^2&\leq \frac{C}{(R-r)^{2j}} \int_{B(x_0,R)\cap\Om}\big(|u_\va|^2+|G|^2\big) + C R^{2m-2j} \int_{B(x_0,R)\cap\Om}|\na^m G|^2 \nonumber\\&
  +C R^{4m-2j+d-\frac{2d}{p}}\Big(\int_{B(x_0,R)\cap\Om} |f^0|^p\Big)^{\frac{2}{p}}, \quad  \text{for }  p>\max\{1,  2d/(d+2m)\}.
\end{align*}

\end{remark}

\section{Convergence rates in Lipschitz domains}
Let $0\leq\rho_\varepsilon \leq 1$ be a function in $C_c^\infty(\Omega)$
 with $ supp(\rho_\varepsilon)\subset \Omega^{3\varepsilon}, \rho_\varepsilon=1$  on $\Omega^{4\varepsilon} $  and $ |\nabla^m \rho_\varepsilon|\leq C \varepsilon^{-m}.$

 \begin{lemma}\label{l3.1}
Suppose that $\Omega$ is a bounded Lipschitz domain in $\mathbb{R}^d$,
and $ A$ satisfies (\ref{cod1})--(\ref{cod3}).
 Let $u_\varepsilon, u_0$ be the weak solutions
 to  Dirichlet problems (\ref{eq1}) and  (\ref{hoeq1}) respectively.
Define  \begin{align}\label{w}
w_\varepsilon = u_\varepsilon-u_0-\varepsilon^m\sum_{|\gamma|=m}\chi^\gamma(x/\va) S^2_\varepsilon(D^\gamma u_0 )\rho_\varepsilon.\end{align}
 Then for any $\phi \in H^m_0(\Omega;\mathbb{R}^n),$ we have
\begin{align}\label{l3.1re1}
 & \Big|\sum_{|\alpha|=|\beta|=m} \int_\Omega D^\alpha \phi_i  A_{ij}^{\alpha\beta}(x/\va)
 D^\beta w_{\varepsilon j}\Big| \nonumber\\
& \leq C \| \nabla^m \phi\|_{L^2(\Omega_{4\varepsilon})}\| \nabla^m u_0\|_{L^{2}(\Omega_{4\varepsilon})}
 + C  \| \nabla^m \phi\|_{L^2(\Omega_{4\varepsilon})} \sum_{ 0\leq k \leq m-1,  }\varepsilon^{k } \| S_\varepsilon(\nabla^{m+k}  u_0 )\|_{L^2(\Omega_{5\varepsilon}\setminus \Omega_{2\varepsilon})}\nonumber\\
 &+ C\| \nabla^m \phi\|_{L^2(\Omega^{2\va};\, \vartheta^{-1})} \|  \nabla^m u_0
 -S_\varepsilon(\nabla^m u_0)  \|_{L^2(\Omega^{2\varepsilon};\, \vartheta)}\nonumber\\
&+  C \|\nabla^m \phi\|_{L^2(\Omega^{2\va};\, \vartheta^{-1})} \sum_{0\leq k\leq m-1}\varepsilon^{m-k } \|  S_\varepsilon(\nabla^{2m-k}  u_0    )\|_{L^2(\Omega^{2\varepsilon};\, \vartheta)},
\end{align}
where $\vartheta(x)=\de(x)$ or $1$, $C$ depends only on $d, n, m, \mu$ and $\Omega.$
\end{lemma}
\begin{proof}
See \cite[Lemma 3.1] {nsx}  for  $\var\equiv1$. The proof for the case $\vartheta(x)=\de(x)$ is almost the same with the help of Lemmas \ref{l2.2},  \ref{l2.3} and  \ref{l2.4}.
\end{proof}

\begin{lemma}\label{l3.2}
Let $\Omega$ be a bounded Lipschitz domain in $\mathbb{R}^d$. Let $A$ satisfy (\ref{cod1})--(\ref{cod3}) and let $u_0$ be the weak solution to the homogenized problem (\ref{hoeq1}) with $\dot{g}
 \in W\!A^{m,2}(\partial\Omega;\mathbb{R}^n) , f\in H^{-m+1}(\Omega;\mathbb{R}^n)$. Then for any $0<\nu< 1/2 ,$
\begin{align}
&\|\na^m u_0\|_{L^2(\Om_{2\va})}\leq C_\nu \varepsilon^{1/2-\nu}  \left\{\|\dot{g}\|_{W\!A^{m,2}(\partial\Omega)}+\|f\|_{H^{-m+1} (\Omega)}\right\},\label{l3.2re1}\\
&\|\na^{m+1} u_0\|_{L^2(\Om^{2\va})}\leq C_\nu \varepsilon^{-1/2-\nu}  \left\{\|\dot{g}\|_{W\!A^{m,2}(\partial\Omega)}+\|f\|_{H^{-m+1} (\Omega)}\right\},\label{l3.2re1'}\\
&\|\na^m u_0\|_{L^2(\Om^{2\va};\, \de^{-1} )}\leq C_\nu \varepsilon^{-\nu }  \left\{\|\dot{g}\|_{W\!A^{m,2}(\partial\Omega)}+\|f\|_{H^{-m+1} (\Omega)}\right\},\label{l3.2re2}\\
&\|\na^{m+1}u_0\|_{L^{2}(\Om^{2\va};\, \de )}\leq C_\nu \varepsilon^{-\nu}  \left\{\|\dot{g}\|_{W\!A^{m,2}(\partial\Omega)}+\|f\|_{H^{-m+1} (\Omega)}\right\},\label{l3.2re3}
\end{align}
where $C_\nu$ depends only on $ d,n,m,\nu,\mu$ and $\Om$.
\end{lemma}
\begin{proof}
Recall that  $f \in H^{-m+1}(\Om)$ can be written as
\begin{align*}
f=\sum_{|\zeta|\leq m-1} D^\zeta f^\zeta ~~\text{ with } ~~ \|f\|_{H^{-m+1}(\Omega)} \approx
 \sum_{|\zeta|\leq m-1} \|f^\zeta\|_{L^{2}(\Omega)}.
\end{align*}
Let $ \widetilde{f}^\zeta $ be the extension of $f^\zeta$, being zero in $\R^d\setminus\Omega.$ Let $\widehat{\Omega} $ be a smooth bounded domain such that $\Om\subset \widehat{\Om}$. Let $v_0$ be the solution to
\begin{align}\label{pt312}
  \mc{L}_0 v_0=\sum_{|\zeta|\leq m-1}D^\ze\widetilde{f}^\ze  \quad \text{in }  \widehat{\Om} ,\quad\quad  Tr(D^\ga v_0)=0 \quad\text{on }  \pa\widehat{\Om} \quad \text{for } 0\leq |\ga|\leq m-1.
  \end{align}
  Standard regularity estimates for higher order elliptic systems give  that  \begin{align}
  &\sum_{1\leq \ell\leq m+1}\|\nabla^{\ell} v_0\|_{L^{2}(\widehat{\Omega})}\leq  C \|f\|_{H^{-m+1}(\Omega)}.    \label{pt315}
  \end{align}
Denote  $\Sigma_t= \{x\in \Om: dist(x,\Omega)=t\}, 0\leq t\leq c_0.$  Similar to \cite[Theorem 3.1]{nsx} (see (3.23) and (3.24) therein), by trace theorem and co-area formula, we may prove that
\begin{align}
 &\sum_{1\leq \ell\leq m} \|\na^\ell v_0\|_{L^2(\Sigma_t)}\leq C \|f\|_{H^{-m+1}(\Omega)},\label{pt316}\\
  &\sum_{1\leq \ell\leq m} \|\na^\ell v_0\|_{L^2(\Omega_\varepsilon)}\leq C\varepsilon^{1/2}
 \|f\|_{H^{-m+1}(\Omega)}.\label{pt317}
\end{align}

On the other hand, setting $u_0(x)=v_0(x)+v(x),$ we have
\begin{align}\label{eqv}
\mathcal{L}_0 v=0 \quad\text{in } \Omega,  \quad   Tr(D^\gamma v)=g_\ga-D^\ga v_0  \quad\text{on }  \pa\Omega\,  \quad\text{for  } 0\leq|\ga|\leq m-1.
\end{align}
Thanks to Theorem  $3'$ and Theorem $5'$ in \cite{agran2007}, we have $ v \in B_2^{m-1/2+s}(\Om)$ for any $ 1/2 \!<\!s\!<\!1$, and,
\begin{align}\label{pt318}
\|v\|_{B^{m-1/2+s}_2(\Omega)} &\leq C_s \big\{\|\dot{g}\|_{W\!A^{m,2}(\partial\Omega)}+\|\dot{v}_0\|_{W\!A^{m,2}(\partial\Omega)} \big\}
\nonumber \\
 &\leq C_s\big\{\|\dot{g}\|_{W\!A^{m,2}(\partial\Omega)}+  \|f\|_{H^{-m+1}(\Omega)}\big\},
\end{align}
where $\dot{v}_0=\big\{D^\ga v_0|_{\pa \Om}\big\}_{|\ga|\leq m-1} $, and (\ref{pt316}) has been  used for the last step. Therefore, we have $D^\al v \in B_2^{s-1/2}(\Om)$ for $|\al|=m$. Thanks to Theorems 1.4.2.4 and 1.4.4.4 in \cite{grisvard},
 \begin{align}\label{pt319}
 \int_\Om |\na^m v(x)|^2 \de(x)^{1-2s}dx\leq C_s \|v\|^2_{B^{m-1/2+s}_2(\Omega)}\leq C_s\big\{\|\dot{g}\|^2_{W\!A^{m,2}(\partial\Omega)}+  \|f\|^2_{H^{-m+1}(\Omega)}\big\}.
\end{align}
This implies that
\begin{align}\label{pt320}
  \int_{\Om_{2\va}} |\na^m v(x)|^2   \, dx &=  \int_{\Om_{2\va}} |\na^m v(x)|^2 \de(x)^{1-2s} \de(x)^{2s-1}\, dx \nonumber\\
 &\leq C_s  \va^{2s-1}\big\{  \|\dot{g}\|^2_{W\!A^{m,2}(\partial\Omega)}+  \|f\|^2_{H^{-m+1}(\Omega)}   \big\}
\end{align} for any $  1/2 \!<\!s\!<\!1$.
By combining (\ref{pt317}) with (\ref{pt320}), we derive  (\ref{l3.2re1}) with $\nu=1-s$.

In view of (\ref{eqv}) and interior estimates for higher order elliptic systems with constant coefficients, we have
\begin{align*}
|\nabla^{m+1} v(x)|\leq \frac{C}{\delta(x)} \Big(\fint_{B(x, \frac{\delta(x)}{8})} |\nabla^{m} v |^2\Big)^{1/2}.
\end{align*}
Thus using (\ref{pt319}) we deduce  that
\begin{align}\label{pt322}
\|\nabla^{m+1} v \|^2_{L^2(\Omega^{2\varepsilon})}
&\leq   C    \int_{\Omega^{2\varepsilon}}\frac{1}{\delta(x)^{3-2s}}
\fint_{B(x, \frac{\delta(x)}{8})} |\nabla^{m} v(y)|^2 \de(y)^{1-2s}dy\, dx \nonumber\\
&\leq C\varepsilon^{2s-3} \| \nabla^m v\|^2_{L^2(\Omega;\, \de^{1-2s})} \nonumber\\
& \leq C_s\varepsilon^{ 2s-3} \left\{  \|\dot{g}\|^2_{W\!A^{m,2}(\partial\Omega)}+  \|f\|^2_{H^{-m+1}(\Omega)}\right\},
\end{align}
which, together with (\ref{pt315}), gives (\ref{l3.2re1'}).

For (\ref{l3.2re2}), it is easy to conclude from (\ref{pt318}) and (\ref{pt319}) that
 \begin{align}\label{pt321}
 \int_{\Om^{2\va}} |\na^m v(x)|^2 \de(x)^{-1}\, dx &=  \int_{\Om^{2\va}} |\na^m v(x)|^2 \de(x)^{1-2s}\de(x)^{2s-2}\, dx\nonumber\\
&\leq C_s \va^{2s-2}\big\{  \|\dot{g}\|^2_{W\!A^{m,2}(\partial\Omega)}+  \|f\|^2_{H^{-m+1}(\Omega)}   \big\}.
\end{align}
On the other hand, by the co-area formula  we deduce that
 \begin{align}\label{pt321'}
 \int_{\Om^{2\va}} |\na^m v_0(x)|^2 \de(x)^{-1}\, dx &= \int_{\Om^{2\va}\setminus\Om^{c_0} } |\na^m v_0(x)|^2 \de(x)^{-1}\, dx +\int_{\Om^{c_0}} |\na^m v_0(x)|^2 \de(x)^{-1}\, dx\nonumber\\
 &=  \int_{2\va}^{c_0}\int_{\Sigma_t} |\na^m v_0(x)|^2 \frac{1}{t} dS dt+ C \int_{\Om^{c_0}} |\na^m v_0(x)|^2 \, dx\nonumber\\
&\leq C  \ln(1/\va)\, \big\{  \|\dot{g}\|^2_{W\!A^{m,2}(\partial\Omega)}+  \|f\|^2_{H^{-m+1}(\Omega)}   \big\}
\end{align}
 for $0<\va<1/2$, where (\ref{pt316}) is used for the last inequality. This, combined with (\ref{pt321}), gives (\ref{l3.2re2}).
The proof for (\ref{l3.2re3}) is the same as (\ref{l3.2re1'}), thus we omit the details.
\end{proof}

\begin{lemma}\label{l3.3}
Suppose that the assumptions of Lemma \ref{l3.2} are satisfied, and $A$ is symmetric, i.e. $A=A^*$. Then for $0<\va< 1/2$,
\begin{align}
&\|\na^m u_0\|_{L^2(\Om_{2\va}  )}\leq C \varepsilon^{1/2}  \big\{\|\dot{g}\|_{W\!A^{m,2}(\partial\Omega)}+\|f\|_{H^{-m+1} (\Omega)}\big\},\label{l3.3re1}\\
&\|\na^m u_0\|_{L^2(\Om^{2\va};\, \de^{-1} )}\leq C\,  [\ln(1/\va)]^{1/2}   \big\{\|\dot{g}\|_{W\!A^{m,2}(\partial\Omega)}+\|f\|_{H^{-m+1} (\Omega)}\big\},\label{l3.3re2}\\
&\|\na^{m+1}u_0\|_{L^{2}(\Om^{2\va};\, \de )}\leq C\, [\ln(1/\va)]^{1/2}  \big\{\|\dot{g}\|_{W\!A^{m,2}(\partial\Omega)}+\|f\|_{H^{-m+1} (\Omega)}\big\},\label{l3.3re3}
\end{align}
where $C $ depends only on $ d,n,m,\mu$ and $\Om$.
\end{lemma}

\begin{proof}
The proof is the same as that of Lemma \ref{l3.2}  except for three places. Firstly, since $A$ is symmetric, in place of   (\ref{pt319}) we have the nontangential maximum function estimate, see e.g.
\cite[Theorem 6.1]{verchota1996},
\begin{align}\label{pt330}
 \|  \mathcal{M}(\nabla^m v)\|_{L^2(\partial \Omega)}
 \leq C\|\dot{v}\|_{W\!A^{m,2}(\partial\Omega)}
 \leq C\big\{\|\dot{g}\|_{W\!A^{m,2}(\partial\Omega)}+  \|f\|_{H^{-m+1}(\Omega)}\big\},
\end{align}
where $\mathcal{M}(\nabla^m v) $ denotes the nontangential maximal function of $\nabla^m v$.
Therefore, instead of (\ref{pt320})  we have
 \begin{align*}
 \int_{\Om_{2\va}} |\na^m v(x)|^2  \, dx
 \leq C \va\left\{ \|\dot{g}\|^2_{W\!A^{m,2}(\partial\Omega)}+  \|f\|^2_{H^{-m+1}(\Omega)}   \right\},
\end{align*}
which, combined with (\ref{pt317}), implies (\ref{l3.3re1}).

Secondly, in substitution for (\ref{pt321})  we  use the nontangential estimate (\ref{pt330}) and the co-area formula to deduce that
\begin{align*}
 \int_{\Om^{2\va}} |\na^m v(x)|^2 \de(x)^{-1}\, dx &= C   \int_{2\va}^{c_0}\int_{\Sigma_t} \frac{1}{t}|\na^m v(x)|^2 dSdt+C\int_{\Om^{c_0}} |\na^m v(x)|^2 \, dx  \nonumber\\
&\leq C\ln(1/\va)\left\{  \|\dot{g}\|^2_{W\!A^{m,2}(\partial\Omega)}+  \|f\|^2_{H^{-m+1}(\Omega)}   \right\},
\end{align*}
which, combined with (\ref{pt321'}), gives (\ref{l3.3re2}).

Finally, instead of (\ref{pt322}), we have
 \begin{align*}
 \|\nabla^{m+1} v \|^2_{L^2(\Omega^{\varepsilon}; \, \de)}
&\leq   C    \int_{\Omega^{\varepsilon}\setminus \Om^{c_0} }\frac{1}{\de(x) }
\fint_{B(x, \frac{\delta(x)}{8})} |\nabla^{m} v(y)|^2  dy\, dx +C \int_{\Omega^{c_0}}
\fint_{B(x, \frac{\delta(x)}{8})} |\nabla^{m} v(y)|^2  dy\, dx\nonumber\\
&\leq   C \int^{c_0}_\va \int_{\Sigma_t} \frac{1}{t}
 |\mathcal{M}(\nabla^{m} v) |^2   dS\, dt + C\left\{ \|\dot{g}\|^2_{W\!A^{m,2}(\partial\Omega)}+  \|f\|^2_{H^{-m+1}(\Omega)}\right\}\nonumber\\
& \leq C \ln (1/\va) \left\{  \|\dot{g}\|^2_{W\!A^{m,2}(\partial\Omega)}+  \|f\|^2_{H^{-m+1}(\Omega)}\right\},
\end{align*}
This, together with (\ref{pt315}), gives (\ref{l3.3re3}). The proof is thus completed.
\end{proof}

\begin{lemma}\label{l3.4}
Assume that $\Omega$ is a bounded Lipschitz domain in $\mathbb{R}^d$ and $A$ satisfies
 (\ref{cod1})--(\ref{cod3}). Let $u_\varepsilon, u_0$ be the weak solutions
 to  Dirichlet problems (\ref{eq1}) and  (\ref{hoeq1}), respectively, with $\dot{g}
 \in W\!A^{m,2}(\partial\Omega;\mathbb{R}^n) , f\in H^{-m+1}(\Omega;\mathbb{R}^n)$. Let $w_\va$ be defined as in (\ref{w}).
Then for any $0<\nu < 1/2 $, we have
\begin{align}\label{l3.4re1}
&\| w_\varepsilon\|_{H^m_0(\Omega)}\leq C_\nu\varepsilon^{1/2-\nu} \left\{\|\dot{g}\|_{W\!A^{m,2}(\partial\Omega)}+\|f\|_{H^{-m+1}(\Omega)}\right\},
\end{align} where $C_\nu$ depends only on $ d, n, m,\nu,\mu$ and $\Om.$
If in addition $ A$ is symmetric, i.e. $A=A^*$, then
\begin{align}\label{l3.4re2}
 \| w_\varepsilon\|_{H^m_0(\Omega)} \leq C \varepsilon^{1/2}  \left\{\|\dot{g}\|_{W\!A^{m,2}(\partial\Omega)}+\|f\|_{H^{-m+1}(\Omega)}\right\},
\end{align}
where $C$ depends only on $d, n, m,\mu$ and $\Om.$
\end{lemma}
\begin{proof} The estimate (\ref{l3.4re2}) has been proved in \cite[Theorem 3.1]{nsx}, we only need to consider (\ref{l3.4re1}) here.
Using Lemmas  \ref{l2.3} and \ref{l2.4}, we deduce from   (\ref{l3.1re1}) that
\begin{align}\label{pl3401}
 &\Big|\sum_{|\alpha|=|\beta|=m} \int_\Omega D^\alpha \phi_i  A_{ij}^{\alpha\beta}(\frac{x}{\varepsilon})
 D^\beta w_{\varepsilon j}\Big|\nonumber \\
 &\leq\, C\Big\{ \| \nabla^m \phi\|_{L^2(\Omega_{4\varepsilon})} \|\nabla^m u_0\|_{L^2(\Omega_{5\va})}+\va\| \nabla^{m+1}  u_0\|_{L^2(\Omega^{2\varepsilon};\, \vartheta)}\|\nabla^m \phi\|_{L^2(\Omega^{2\va};\, \vartheta^{-1})} \Big\}.
\end{align}
Taking $\var=1, \phi=w_\va$ and using the ellipticity condition (\ref{cod1}), we obtain that
\begin{align*}
 \|\na^m w_{\varepsilon}\|_{H_0^m(\Om)}
\leq\, C\Big\{ \|\nabla^m u_0\|_{L^2(\Omega_{5\va})}+\va\| \nabla^{m+1}  u_0\|_{L^2(\Omega^{2\varepsilon})} \Big\},
\end{align*}
from which and (\ref{l3.2re1}), (\ref{l3.2re1'}), we obtain (\ref{l3.4re1}) immediately.
\end{proof}

We are now prepared to prove Theorem \ref{thcon}.
\begin{proof}[\bf Proof of Theorem \ref{thcon}]  We only provide the details for (\ref{thconre1}), as the proof of (\ref{thconre2}) is similar.
By scaling, we may assume that $$\|\dot{g}\|_{W\!A^{m,2}(\partial\Omega)}+\|f\|_{H^{-m+1}(\Omega)}=1.$$
 For any fixed $F\in H^{-m+1}(\Omega;\mathbb{R}^n)$,
  let $\psi_\varepsilon\in H^m_0(\Omega;\mathbb{R}^n)$  be the weak solution to the Dirichlet problem
\begin{equation*}
 \begin{cases}
 \mathcal{L}^*_\varepsilon \psi_\varepsilon =F  &\text{ in } \Omega,  \\
 Tr (D^\gamma \psi_\varepsilon)=0   & \text{ on } \partial\Omega \ \ \text{ for } 0\leq|\gamma|\leq m-1,
\end{cases}
\end{equation*}
and let $\psi_0\in H^m_0(\Omega;\mathbb{R}^n)$ be the weak solution to the homogenized problem
 \begin{equation*}
 \begin{cases}
 \mathcal{L}_0^* \psi_0 =F  &\text{ in } \Omega,  \\
 Tr (D^\gamma \psi_0)=0  & \text{ on } \partial\Omega \ \  \text{ for  } 0\leq|\gamma|\leq m-1.
\end{cases}
\end{equation*}
Here $\mc{L}^*_\va$ and $\mc{L}^*_0$ are the adjoint operators of $\mc{L}_\va$ and $\mc{L}_0$ respectively.
Let $0\leq\widetilde{\rho}_\varepsilon \leq 1$ be a function in $C_c^\infty(\Omega)$
 with $ supp(\widetilde{\rho}_\varepsilon)\subset \Omega^{6\varepsilon}, \widetilde{\rho}_\varepsilon=1$  on $\Omega^{8\varepsilon} $  and $ |\nabla^m \widetilde{\rho}_\varepsilon|\leq C \varepsilon^{-m}.$ Set $$ \Psi_\varepsilon=\psi_\varepsilon-\psi_0-\varepsilon^m\sum_{|\gamma|=m}\chi^{*\gamma}(\frac{x}{\varepsilon}) S^2_\varepsilon(D^\gamma  \psi_0 )\widetilde{\rho}_\varepsilon .$$
It satisfies the same properties as $w_\va,$ since $A^*$ satisfies the same properties as $A$. Note that $w_\varepsilon\in H^m_0(\Omega;\mathbb{R}^n)$,  we  deduce that
\begin{align}\label{pt100}
\langle F,  w_\varepsilon \rangle_{ H^{-m+1}(\Omega)\times H_0^{m-1}(\Omega) }
 &=   \sum_{|\alpha|=|\beta|=m} \int_{\Omega} A^{\be \al}(\frac{x}{\varepsilon}) D^\alpha w_\varepsilon
 D^\beta  \Psi_\varepsilon
 + \sum_{|\alpha|=|\beta|=m}\int_{\Omega}
 A^{\be \al}(\frac{x}{\varepsilon})D^\al w_\varepsilon D^\beta\psi_0  \nonumber\\
 &\quad + \sum_{|\alpha|=|\beta|=m}\int_{\Omega}
 A^{\be \al}(\frac{x}{\varepsilon}) D^\alpha w_\varepsilon D^\beta
 \Big\{\sum_{|\gamma|=m}\varepsilon^m\chi^{*\gamma}(\frac{x}{\varepsilon}) S^2_\varepsilon(D^\gamma  \psi_0 )\widetilde{\rho}_\varepsilon  \Big\} \nonumber \\
 &\doteq J_1+J_2+J_3.
\end{align}
By (\ref{l3.4re1}),  we obtain that
\begin{align*}
|J_1|\leq C\| w_\varepsilon\|_{H^m_0(\Omega)}\| \Psi_\varepsilon\|_{H^m_0(\Omega)}\leq C_\nu \va^{1-2\nu}
   \|F\|_{H^{-m+1}(\Omega)}.
\end{align*}
Using (\ref{pl3401}) and taking $\var(x)=\de(x)$, we have
\begin{align}\label{j2}
|J_2|
  &\leq C  \varepsilon \|\nabla^m \psi_0 \|_{L^2(\Omega^{2\va};\, \de^{-1})} \|\na^{m+1}u_0\|_{L^2(\Omega^{2\va};\, \de)}+ C\|\na^mu_0\|_{L^{2}(\Omega_{5\va})} \|\nabla^m \psi_0\|_{L^2(\Omega_{4\varepsilon})}   .
\end{align}
By  (\ref{l3.2re1}) (note that $\psi_0$ also satisfies  (\ref{l3.2re1})), we get
\begin{align*}
\|\nabla^{m}u_0 \|_{L^2(\Omega_{5\varepsilon})} \|\nabla^{m}\psi_0 \|_{L^2(\Omega_{4\varepsilon})}\leq  C_\nu\varepsilon^{1-2\nu}\|F\|_{H^{-m+1}(\Omega)}.
\end{align*}
Furthermore, taking (\ref{l3.2re2}) and (\ref{l3.2re3}) into consideration, we  conclude from (\ref{j2}) that
\begin{align*}
|J_2| \leq C_\nu\varepsilon^{1-2\nu}\|F\|_{H^{-m+1}(\Omega)}.
\end{align*}

We now turn to $J_3$. By (\ref{pl3401}), we obtain that
 \begin{align}\label{ptc004}
  |J_3 |\leq& C \varepsilon\sum_{|\ga|=m} \|\na^{m+1}u_0\|_{L^2 (\Omega^{2\va};\, \de)}
\varepsilon^{m}\|\nabla^m\big\{ \chi^{*\ga}(\frac{x}{\varepsilon})
  S^2_\varepsilon(D^\ga \psi_0  )\widetilde{\rho}_\varepsilon\big\}\|_{L^2(\Omega^{2\va};\, \de^{-1})} \nonumber\\
   &+ C\sum_{|\ga|=m} \varepsilon^m\|\na^m u_0\|_{L^{2}(\Omega_{5\varepsilon})}\|\nabla^m \big\{ \chi^{*\ga}(\frac{x}{\varepsilon}) S^2_\varepsilon(D^\ga \psi_0    )\widetilde{\rho}_\varepsilon\big\}\|_{L^2(\Omega_{4\varepsilon}
   )},
\end{align}
where the last term is zero by the definition of $\wt{\rho}_\va.$
To estimate the first term, we note that
\begin{align*}
&\,\varepsilon^m\|D^\be\big\{ \chi^{*\ga}(\frac{x}{\varepsilon})
  S^2_\varepsilon(D^\ga  \psi_0  )\widetilde{\rho}_\varepsilon\big\}\|_{L^2(\Omega^{2\va};\, \de^{-1})}\nonumber\\
&\leq C \|(D^\be \chi^{*\gamma})(\frac{x}{\varepsilon}) S^2_\varepsilon(D^\ga  \psi_0 )\widetilde{\rho}_\varepsilon\|_{L^2(\Omega^{2\va};\, \de^{-1})}
 \nonumber\\
 &\ \ +C \varepsilon^m \|\chi^{*\gamma}(\frac{x}{\varepsilon}) S^2_\varepsilon(D^{\be+\ga}  \psi_0 )\widetilde{\rho}_\varepsilon\|_{L^2(\Omega^{2\va};\, \de^{-1})}
+ C\varepsilon^m \|\chi^{*\gamma}(\frac{x}{\varepsilon}) S^2_\varepsilon(D^\ga  \psi_0 )D^\be\widetilde{\rho}_\varepsilon\|_{L^2(\Omega^{2\va};\, \de^{-1})}\nonumber\\
 &\ \ +C\sum_{\substack{|\zeta+\eta+\xi|=m\\ 1\leq |\zeta|, |\eta|, |\xi|  }}
 \varepsilon^{|\eta|+|\xi|}
 \|(D^\zeta\chi^{*\gamma})(\frac{x}{\varepsilon}) S^2_\varepsilon(D^{\ga+\eta}  \psi_0  )D^\xi\widetilde{\rho}_\varepsilon\|_{L^2(\Omega^{2\va};\, \de^{-1})}\nonumber\\
 &\doteq J_{31}+J_{32}+J_{33} +J_{34},
\end{align*}
for all $|\be|=|\ga|=m$. By Lemmas  \ref{l2.2} and \ref{l3.2}, we obtain that
 \begin{align*}
&J_{31}\leq C  \|  S_\varepsilon(\nabla^m  \psi_0 ) \|_{L^2(\Omega^{6\va};\, \de^{-1})}
\leq C\| \nabla^m {\psi}_0 \|_{L^2(\Omega^{5\va};\, \de^{-1})} \leq C_\nu \va^{-\nu} \|F\|_{H^{-m+1}(\Omega)},\\
&J_{33}\leq C \|  S_\varepsilon(\nabla^m  \psi_0  )\|_{L^2(\Omega_{9\varepsilon}\setminus \Omega_{4\varepsilon};\, \de^{-1})}
\leq C_{\nu} \va^{-\nu} \|F\|_{H^{-m+1}(\Omega)}.
\end{align*}
Furthermore, by Lemmas  \ref{l2.2}, \ref{l2.3} and \ref{l3.2}, we see that
 \begin{align*}
&\quad\quad J_{32}\leq   C  \varepsilon^m \|   S_\varepsilon(\nabla^{2m}  \psi_0 ) \|_{L^2(\Omega^{4\varepsilon};\, \de^{-1})} \leq  C_\nu\va^{-\nu} \|F\|_{H^{-m+1}(\Omega)},\\
 J_{34} &= C\sum_{\substack{k_1+k_2+k_3=m\\ 1\leq k_i ,i=1,2,3}}\varepsilon^{k_2+k_3}
 \| (\nabla^{k_1}\chi^{*})(\frac{x}{\varepsilon})  S^2_\varepsilon(\nabla^{k_2+m}
  \psi_0 ) \nabla^{k_3}\widetilde{\rho}_\varepsilon\|_{L^2(\Omega_{8\varepsilon}\setminus\Omega_{6\varepsilon};\, \de^{-1})}\\
& \leq C\sum_{1\leq k_2\leq m-2} \varepsilon^{k_2 }\|  S_\varepsilon(\nabla^{k_2+m} \psi_0 )\|_{L^2(\Omega_{9\varepsilon}\setminus\Omega_{5\varepsilon};\, \de^{-1})}\\
& \leq C_\nu \va^{-\nu} \|F\|_{H^{-m+1}(\Omega)}.
 \end{align*}
Taking the estimates on $J_{31}, J_{32}, J_{33}, J_{34}$ into (\ref{ptc004}), and using (\ref{l3.2re3}), we obtain that
\begin{align*}
 |J_3|\leq C_\nu \va^{1-2\nu}   \|F\|_{H^{-m+1}(\Omega)}.
 \end{align*}
In view of the estimates on $J_1, J_2, J_3$ and (\ref{pt100}), we have proved that
 \begin{align*}
 \Big| \langle F, w_\varepsilon\rangle_{H^{-m+1}(\Omega)\times H_0^{m-1}(\Omega) } \big|
  \leq C_\nu \va^{1-2\nu}   \|F\|_{H^{-m+1}(\Omega)},
  \end{align*}
  which, combined with the following estimate
\begin{align*}
\|\varepsilon^m\sum_{|\gamma|=m}\chi^\gamma(x/\va) S^2_\varepsilon(D^\gamma u_0 )\rho_\varepsilon\|_{H_0^{m-1}(\Om)}\leq C\va,
\end{align*}
  gives (\ref{thconre1}). The proof is complete.
  \end{proof}
\begin{remark}\label{remark1} 

Part of our motivation for the proof of (\ref{thconre1}) is the finding that $u_0$ satisfying certain weighted estimates such as (\ref{l3.2re1})--(\ref{l3.2re3}),  which give  a proper control on the solution $u_0$ in $\Om^\va$ and $\Om_\va$. This also inspires us to modify the duality method with weight $\de(x)$.
  We mention that weight  functions have been used previously in \cite{klsal2, xusiam2016} to derive the suboptimal $O(\va\ln(1/\va))$ convergence rate for second order elliptic systems with symmetric coefficients. Our consideration on the suboptimal convergence rate is also in debt to these works.
\end{remark}

\section{Uniform $C^{m-1, \lm}$ estimates}

In this section, we consider uniform boundary $C^{m-1,\lm}, 0<\lm<1,$ estimates of $u_\va$ in $C^1$ domains. Throughout the section, we always assume that $A$ satisfies (\ref{cod1}) and (\ref{cod3}). Recall that locally the boundary of a $C^1$ domain is the graph of a $C^1$ function, we thus restrict our considerations to equations on $(D_r, \De_r)$ defined in (\ref{drder}) with the defining function satisfying (\ref{c1c}).
Let
$$\mf{P}_k=\Big\{(P^1_k, P^2_k,..., P^n_k)\mid P^i_k, 1\leq i\leq n, \text{ are polynomials with real coefficients of degree } k  \Big\}.
$$
Let $u_\va\in H^m(D_{2r}; \mb{R}^n )$ be a weak solution to
\begin{align}\label{uva1}
 \mc{L}_\va u_\va= F  \quad\text{in } D_{2r}, \quad\quad
 Tr (D^\ga u_\va)=D^\ga G  \quad\text{on }  \De_{2r} \quad\text{for  } 0\leq|\ga|\leq m-1,
  \end{align}
 where $G\in C^{m-1,1}(D_{2r};\R^n), F\in L^p(D_{2r}; \R^n)$ with $ p> \max\{1,  2d/(d+2m-2)\}$.  Define
\begin{align}\label{l312}
 \Phi_\lm(r,u_\va)=
\frac{1}{r^{m-1+\lm}}\inf_{P_{m-1}\in \mf{P}_{m-1}} \bigg\{&\Big(\fint_{D_{r}}|u_\va-P_{m-1}|^2 \Big)^{1/2}+ r^{2m} \Big(\fint_{D_{r}} |F|^{p}\Big)^{1/p} \nonumber\\
 &  + \sum_{j=0}^{m}  r^j  \|\na^j (G-P_{m-1})\|_{L^\infty(D_r)}
  \bigg\},\quad \quad  0<\lm<1 .
\end{align}
\begin{lemma}\label{lemm3.1}
Let $0<\va\leq r\leq1$ and $ \Phi_\lm(r,u_\va)$ be defined as above.
There exists $u_0\in H^m(D_r; \mb{R}^n )$ such that $\mc{L}_0 u_0=F$ in $D_r$, $Tr (D^\ga u_0)=D^\ga G $ on $\De_r$ for $ 0\leq|\ga|\leq m-1$, and
\begin{align}\label{l311}
\Big(\fint_{D_r}|u_\va-u_0|^2 \Big)^{1/2}\leq Cr^{m-1+\lm} \Big(\frac{\va}{r}\Big)^{1/4} \Phi_\lm(2r,u_\va),
\end{align}
where $C$  depends only on $d, n,m, p, \mu$ and $M$ in (\ref{c1c}).

\end{lemma}
\begin{proof} Let us first assume that $r=1$.
By Caccioppoli's inequality (see Remark \ref{remark0}), we have
\begin{align}\label{pl311}
  \|u_\va\|_{H^m( D_{3/2})}  \leq C   \Big\{\|u_\va\|_{L^2(D_2)}+ \|F\|_{L^{p}(D_2)}+  \sum_{j=0}^{m}  \|\na^j G\|_{L^2(D_2)} \Big\},
\end{align} for $p>\max\{1,  2d/(d+2m)\}$.
By the co-area formula, there exist $t\in [5/4, 3/2]$ such that
\begin{align}\label{pl312}
\|u_\va\|_{H^m(\pa D_t\setminus\De_2) } \leq   C   \Big\{\|u_\va\|_{L^2(D_2)}+ \|F\|_{L^{p}(D_2)} +\sum_{j=0}^{m}  \|\na^j G\|_{L^2(D_2)} \Big\},
\end{align}
where $C$ depends only on $d,n, m,\mu$.
Now let $u_0$ be the weak solution to $$\mc{L}_0u_0=F ~\text{ in } D_t, ~~~~~~ Tr (D^\ga u_0)=  Tr(D^\ga u_\va)  \text{ on } \pa D_t.$$
Note that $F\in L^p\hookrightarrow H^{-m+1}$ when $p>\max\{ 2d/(d+2m-2), 1\}$.   As a consequence of (\ref{thconre1}) in Theorem \ref{thcon}, we have
\begin{align*}
\|u_\va-u_0\|_{L^2(D_t)} &\leq C \va^{1/4} \Big\{  \|u_\va\|_{H^m(\pa D_t)}+  \|F\|_{L^{p}(D_2)}  \Big\},
\end{align*}
where $C $ depends only on $d,n,m,\mu$ and $M$ in (\ref{c1c}).
This, together with (\ref{pl312}), yields  \begin{align}\label{pl313}
 \|u_\va-u_0\|_{L^2(D_1)}
 \leq \|u_\va-u_0\|_{L^2(D_t)}
 \leq C  \va^{1/4}
  \Big\{\|u_\va\|_{L^2(D_2)}+ \|F\|_{L^{p}(D_2)}+   \| G\|_{C^{m-1,1}(D_2)} \Big\},
\end{align} for $p>\max\{1, 2d/(d+2m-2) \}.$

We now perform scaling for general $0<r<1$. Set  $$v_\va (x)=u_\va(r x),\quad \wt{G}(x)=G(rx), \quad \wt{F}(x) =r^{2m}F(rx).$$
By (\ref{uva1}),  we know that
$
\mc{L}_\frac{\va }{r} v_\va =\wt{F}(x)$  in $\widetilde{D}_2,$  and $
 Tr (D^\ga  v_\va)=D^\ga\wt{G}(x)$ on $\widetilde{\De}_2$  for $0\leq|\ga|\leq m-1,$ where
 \begin{align*}
 \widetilde{D}_2=\big\{ (x',x_d) \in \R^d: |x'|  < 2 \text{ and } \psi_r(x')< x_d<\psi_r(x')+2 \big\},\\
  \widetilde{\De}_2 =\big\{ (x',\psi_r(x')) \in \R^d: |x'|  < 2  \big\}, \text{ with }   \psi_r(x')=r^{-1}\psi(rx').
 \end{align*}

 Thanks to (\ref{pl313}), there exists $v_0$  with
$\mc{L}_0 v_0=\wt{F}(x)$ in $\widetilde{D}_1$, $Tr (D^\ga v_0)=Tr (D^\ga  v_\va)$  on $\pa \widetilde{D}_1$ for  $ 0\leq|\ga|\leq m-1,$
such that \begin{align*}
 \|v_\va-v_0\|_{L^2(\widetilde{D}_1)}
 &\leq C \Big(\frac{\va}{r}\Big)^{1/4} \Big\{\|v_\va\|_{L^2(\widetilde{D}_2)}+   \|\wt{F}(x)\|_{L^{p}(\widetilde{D}_2)} +\| \wt{G}\|_{C^{m-1,1}(\widetilde{D}_2)} \Big\}.
\end{align*}
Setting $u_0(x)=v_0(x/r)$,  we  then obtain by the change of variables,
 \begin{align*}
\Big(\fint_{D_r}|u_\va-u_0|^2 \Big)^{1/2}
 \leq C  \Big(\frac{\va}{r}\Big)^{1/4} \bigg\{&\Big(\fint_{D_{2r}}|u_\va |^2 \Big)^{1/2} + r^{2m} \Big(\fint_{D_{2r}} |F|^{p}\Big)^{1/p}
 + \sum_{j=0}^{m}  r^j \|\na^j G\|_{L^\infty(D_{2r})}
  \bigg\}.
 \end{align*}
Note that the above inequality still hold if we subtract a polynomial $P_{m-1}\in \mf{P}_{m-1}$ from $u_\va, u_0$ and $G$  simultaneously. This gives (\ref{l311}) by taking the infimum with respect to $P_{m-1}$.
\end{proof}

\begin{lemma}\label{lemm3.2}
For  $0<\va\leq r\leq1$, let  $u_0\in H^m(D_2; \mb{R}^n)$ be a weak solution to
   \begin{align*}
 \mc{L}_0 u_0=F    \quad\text{in } D_2,   \quad\quad
 Tr (D^\ga u_0)=D^\ga G   \quad\text{on }  \De_2  \quad\text{for  } 0\leq|\ga|\leq m-1,
 \end{align*}
  where $G\in C^{m-1,1}( D _2; \R^n), F\in L^p(D_2; \R^n)$  with $p>\max\{ d/(m+1),  2d/(d+2m-2), 1\}$. Then for any $ 0<\lm < \min\{m+1-d/p, 1\}$, there exist $  \lm_0<\min\{m+1-d/p, 1\}$ and a constant $C$  depending only on $d,n,m, p, \mu, M$ and $\tau(t)$ in (\ref{c1c}), such that
\begin{align}\label{l321}
\Phi_\lm(\de r;u_0)\leq C \de^{\lm_0-\lm } \Phi_\lm(r;u_0) \quad\text{for }\,  0<\de<1/4.
\end{align}

\end{lemma}
\begin{proof}
By rescaling, we assume that $r=1$. For $ 0< \lm < \min\{m+1-d/p, 1\}$, fix $\lm_0$ such that $ \lm<\lm_0<\min\{m+1-d/p,   1\}.$ Set $P_{m-1}$ in (\ref{l312}) as  $$P_{m-1}(x)= \sum_{|\al|=0}^{m-1} \frac{1}{\al!} D^\al u_0(0) x^{\al}=\sum_{|\al|=0}^{m-1} \frac{1}{\al!} D^\al G(0) x^{\al}.$$  It is not difficult to find that
\begin{align}\label{pl321}
 \Phi_\lm(\de,u_0)\leq  C \de^{\lm_0-\lm} \|u_0\|_{C^{m-1,\lm_0}(D_\de)}
 &+  \de^{m+1-\lm-d/p} \Big( \fint_{D_{1}} |F|^{p}\Big)^{1/p}
 +  \de^{1-\lm}   \sum_{j=0}^m \|\na^jG\|_{L^\infty(D_1)}
\end{align} for  any $0<\de<1/4$.
Observe that when $p>\max\{d/(m+1),1\}$, $L^p(D_2;\R^n)\hookrightarrow W^{-m,q}(D_2;\R^n)$ for some $q>d.$ Combing  the $C^{m-1,\lm_0}$ estimate for higher order  elliptic systems with constant coefficients in $C^1$ domains (see e.g., \cite{dk, br}) and a localization argument (see e.g., the proof of Corollary \ref{co4.1}),
we have
\begin{align}\label{pl322}
 \|u_0\|_{C^{m-1,\lm_0}(D_\de)}&\leq C\|u_0\|_{C^{m-1,\lm_0}(D_{1/4})}\nonumber\\
 &\leq C\bigg\{ \Big(\fint_{D_1} |u_0|^2\Big)^{1/2}  +\Big(\fint_{D_1} |F|^p\Big)^{1/p}
 +  \sum_{j=0}^m \|\na^jG\|_{L^\infty(D_1)} \bigg\}.
\end{align}
Taking (\ref{pl322}) into (\ref{pl321}),  we  derive that
\begin{align*}
\Phi_\lm(\de,u_0)\leq  C \de^{\lm_0-\lm} \bigg\{ \Big(\fint_{D_1} |u_0|^2\Big)^{1/2}  +\Big(\fint_{D_1} |F|^p\Big)^{1/p}   +\sum_{j=0}^m \|\na^jG\|_{L^\infty(D_1)}  \bigg\}.
\end{align*}
 Substituting $u_0, G$ by $u_0-P_{m-1}$ and $ G-P_{m-1}$ respectively and taking the infimum with respect to $P_{m-1}\in\mf{P}_{m-1}$, we obtain (\ref{l321}) immediately.
\end{proof}

\begin{lemma}\label{lemm3.3}
For $0<\va\leq r\leq1/2$, let $\Phi_\lm(r,u_\va)$ be defined as in (\ref{l312}). Then there exists $\de\in (0, 1/4)$ depending only on $d,n,m, p,\lm,\mu, M$ and $\tau(t)$ in (\ref{c1c}), such that
\begin{align}\label{l331}
\Phi_\lm(\de r;u_\va)\leq \frac{1}{2}\Phi_\lm(2r;u_\va)+C\Big(\frac{\va}{r}\Big)^{1/4}\Phi_\lm(2r;u_\va),
\end{align}
where $C$ depends only on $d, n, m, p,  \lm, \mu, M$ and $\tau(t)$ in (\ref{c1c}).
\end{lemma}
\begin{proof}
By the definition, it is easy to find that
\begin{align*}
\Phi_\lm(\de r;u_\va)&\leq \Phi_\lm(\de r;u_0)+\frac{1}{(\de r)^{m-1+\lm}}\Big(\fint_{D_{\de r}}|u_\va-u_0|^2 \Big)^{1/2}\nonumber\\
& \leq C \de^{\lm_0-\lm}\Phi_\lm(r;u_0)+\frac{1}{(\de r)^{m-1+\lm}}\Big(\fint_{D_{\de r}}|u_\va-u_0|^2 \Big)^{1/2}\nonumber\\
& \leq C \de^{\lm_0-\lm}\Phi_\lm(r;u_\va) +\frac{C\de^{\lm_0-\lm}}{r^{m-1+\lm}}\Big(\fint_{D_{r}}|u_\va-u_0|^2 \Big)^{1/2}  +\frac{1}{(\de r)^{m-1+\lm}}\Big(\fint_{D_{\de r}}|u_\va-u_0|^2 \Big)^{1/2}\nonumber\\
&\leq C \de^{\lm_0-\lm}\Phi_\lm(2r;u_\va) +\frac{C_\de}{r^{m-1+\lm}}\Big(\fint_{D_{r}}|u_\va-u_0|^2 \Big)^{1/2}.
\end{align*}
Taking $\de$ small such that $ C \de^{\lm_0-\lm}<1/2$, and then using Lemma \ref{lemm3.1},  we obtain (\ref{l331}) directly.
 \end{proof}

\begin{proof}[\bf Proof of Theorem \ref{thholder}]
We only need to consider the case $\va\leq r<1/4$, since the estimate (\ref{tholdere1}) is trivial when $1/4\leq r\leq 1$, following directly from  Caccioppoli's  inequality.
Thanks to Lemma \ref{lemm3.3}, we can take $N_0$ large enough such that
\begin{align}\label{t311}
\Phi_\lm(\de r;u_\va)\leq \frac{1}{2}\Phi_\lm(2r;u_\va)+C\Big(\frac{1}{N_0}\Big)^{1/4}\Phi_\lm(2r;u_\va)\leq\Phi_\lm(2r;u_\va),
\end{align} for $ r \geq N_0\va  ,$
where $\de$ given by  Lemma \ref{lemm3.3} is fixed.
Hence, by iteration we have \begin{align}\label{t311'}\Phi_\lm(r;u_\va)\leq C \Phi_\lm(1;u_\va)  \quad\text{ for } r\in[N_0\va, 1/2).\end{align}  On the other hand, for $ \va\leq r<N_0\va,$
it is obvious that  $$\Phi_\lm(r;u_\va)\leq C \Phi_\lm(N_0\va;u_\va)\leq C \Phi_\lm(1;u_\va),$$ where $C$ depends on $N_0$. This, together with (\ref{t311'}), gives
\begin{align}\label{pt312}
 \Phi_\lm(r;u_\va)\leq C\Phi_\lm(1;u_\va)  \quad\text{ for } \va\leq r\leq 1/2.
\end{align}
By Caccioppoli's inequality, we deduce that
\begin{align*}
\Big( \fint_{D_r}|\na^m (u_\va-P_{m-1})|^2\Big)^{1/2}
& \leq Cr^{-m}\inf_{P_{m-1}\in \mf{P}_{m-1}}\bigg\{\Big(\fint_{D_{2r}}|u_\va-P_{m-1} |^2 \Big)^{1/2} \nonumber\\
&\quad+ r^{2m} \Big(\fint_{D_{2r}} |F|^{p}\Big)^{1/p}
  + \sum_{j=0}^{m} r^j \Big( \fint_{D_{2r}} |\na^j (G-P_{m-1})|^{2}\Big)^{1/2}
  \bigg\}.\\
&=Cr^{\lm-1} \Phi_\lm(2r, u_\va)
\leq  C r^{\lm-1}  \Phi_\lm(1, u_\va)\\
 &\leq C r^{\lm-1}\Big\{ \Big(\fint_{D_1}|u_\va|^2 \Big)^{1/2}+\Big(\fint_{D_1} |F|^p\Big)^{1/p}   +\sum_{j=0}^{m}  \|\na^jG\|_{L^\infty(D_1)} \Big\},
\end{align*}
for all $\va\leq r<1/2$ and any $P_{m-1}\in \mf{P}_{m-1},$ which is exactly (\ref{tholdere1}).
\end{proof}

\begin{corollary} \label{co3.1}
 In addition to the assumptions of Theorem \ref{thholder}, if $A\in V\!M\!O (\mathbb{R}^d),$ then for any $0<\lm<\min\{m+1-d/p,   1\}$,
\begin{align}\label{co311}
\|u_\va\|_{C^{m-1,\lm}(D_{1/4})}\leq C \Big\{\Big(\fint_{D_1}|u_\va|^2 \Big)^{1/2}& +\Big(\fint_{D_1} |F|^p\Big)^{1/p}   + \|G\|_{C^{m-1,1}(D_1)}\Big\},
\end{align}
where $C$ depends only on $d,n,m, p,  \mu$ as well as $ M, \tau(t)$ in (\ref{c1c}) and $\varrho(t)$ in (\ref{vmo}).
\end{corollary}
\begin{proof}
  It is enough to assume $0<\va< 1/2 $, as the other case is trivial. Setting $$v_\va(x)=u_\va(\va x), ~~~ \wt{F}(x) =\va^{2m} F(\va x), ~~~\wt{G}(x)=G(\va x),$$ then $v_\va$ satisfies
\begin{align}\label{pc311}
\mc{L}_1 v_\va =\wt{F}(x)   \quad\text{ in } D_1,  \quad\quad
 Tr (D^\ga  v_\va)=D^\ga\wt{G}(x)  \quad\text{ on }  \De_1\quad\text{for } 0\leq|\ga|\leq m-1.
\end{align}
By $ C^{m-1,\lm}$ estimates for operator $\mc{L}_1$ in $C^1$ domains (\cite{br,dk}) and a localization argument,  we have
for any $0<\lm <\min\{m+1-d/p, 1\}$ and $0<s< 1/2 ,$
\begin{align}\label{pc312}
\Big(\fint_{D_s}|\na^m v_\va|^2\Big)^{1/2}\leq C s^{\lm-1}  \Big\{ &\Big(\fint_{D_1} |v_\va|^2\Big)^{1/2}  +\Big(\fint_{D_1} |\wt{F}|^p\Big)^{1/p}
 + \sum_{j=0}^{m} \|\na^j\wt{G}\|_{L^\infty(D_1)} \Big\}.
\end{align}
By the change of variables, we obtain for $0<r< \va/2 ,$
\begin{align}\label{pc313}
\Big(\fint_{D_r}|\na^m u_\va|^2\Big)^{1/2}\leq C \Big( \frac{r}{\va}\Big)^{\lm-1} & \frac{1}{\va^m}  \Big\{\Big(\fint_{D_\va}| u_\va|^2\Big)^{1/2}\nonumber\\
&+\va^{2m}\Big( \fint_{D_\va}|F|^p\Big)^{1/p}
 + \sum_{j=0}^{m} \va^j \|\na^jG\|_{L^\infty(D_\va)}\Big\} .
\end{align}
Subtracting $P_{m-1}$ from $u_\va$ and $G$ simultaneously,  and taking (\ref{pt312}) in consideration, we obtain that
 \begin{align}\label{pc314}
\Big(\fint_{D_r}|\na^m u_\va|^2\Big)^{1/2}\leq C r^{\lm-1}  \Big\{&\Big(\fint_{D_1}|u_\va|^2 \Big)^{1/2} +\Big(\fint_{D_1} |F|^p\Big)^{1/p}  +\sum_{j=0}^{m} \|\na^jG\|_{L^\infty(D_1)}\Big\}
\end{align}for any $0<r\leq \va.$
In view of (\ref{tholdere1}), we know that (\ref{pc314}) holds for $ 0\leq r<  1/2 $.
Combining  (\ref{pc314}) with similar interior $C^{m-1,\lm}$ estimate  in \cite[Corollary 5.1]{nsx}, we obtain that
\begin{align*}
\Big(\fint_{B(x,r)\cap D_{1/4}}|\na^m u_\va|^2\Big)^{1/2}\leq C r^{\lm-1}  \Big\{&\Big(\fint_{D_1}|u_\va|^2 \Big)^{1/2} +\Big(\fint_{D_1} |F|^p\Big)^{1/p} +\sum_{j=0}^{m}  \|\na^jG\|_{L^\infty(D_1)}\Big\}
\end{align*}for any $0<r<r_0$ ($r_0$ is small) and $x\in D_{1/4}$.
This gives (\ref{co311}) by the Campanato  characterization of H\"{o}lder spaces.
\end{proof}
\begin{remark}
Under the assumptions of Corollary \ref{co3.1}, if $F,\, G\equiv 0$ we may use Poincar\'{e}'s inequality to deduce from (\ref{pc313}) that
\begin{align}\label{remark21}
\Big(\fint_{D_r}|\na^m u_\va|^2\Big)^{1/2}\leq C \Big( \frac{r}{\va}\Big)^{\lm-1}   \Big(\fint_{D_\va}| \na^m u_\va|^2\Big)^{1/2}
\end{align}   for any $0<r\leq\va.$ This will be used to establish the uniform $W^{m,p}$ estimate in the next section.
\end{remark}

\section{Uniform $W^{m,p}$ estimates }

This section is devoted to the uniform $W^{m,p}$ estimate for $u_\va$ in $C^1$ domains under the assumption $A\in V\!M\!O(\R^d)$.

\begin{lemma}\label{lemm4.1}
Assume that the coefficient matrix $A$ and the domain $\Om$ satisfy the assumptions of Theorem \ref{twmp}. Let $B(x_0,r), r<r_0,$ be a ball centered at $x_0\in \pa\Om$ with radius $r$. Suppose that $u_\va \in H^m(2B\cap\Om; \R^n)$ is a weak solution to
$$\mc{L}_\va u_\va=0 \quad\text{in } B(x_0,2r)\cap\Om, \quad\quad  Tr(D^\ga u_\va)=0 \quad \text{on }  B(x_0, 2r) \cap\pa\Om \quad\text{for }\,  0 \leq |\ga|\leq m-1.$$
Then for any $2\leq p<\infty$,
\begin{align}\label{l411}
\Big(\fint_{B(x_0,r)\cap\Om} |\na^m u_\va|^p\Big)^{1/p}  \leq  C \Big(\fint_{B(x_0,2r)\cap \Om } |\na^m u_\va |^2\Big)^{1/2},
\end{align}
where $C$ depends only on $d,n,m,p,\mu$ as well as $ M, \tau(t)$ in (\ref{c1c}) and $\varrho(t)$ in (\ref{vmo}).
\end{lemma}
\begin{proof}
We only need to consider the case $\va<\frac{1}{4}$. Since if else $A(x/\va)$ satisfies (\ref{vmo}) uniformly, and (\ref{l411}) follows from the existing $W^{m,p}$ estimates for higher order  elliptic systems with $V\!M\!O$ coefficients, see e.g., \cite{br,dk}. Also, note that the function $\psi_r(x')=r^{-1}\psi(rx')$ satisfies condition (\ref{c1c}) uniformly. We can then fix our considerations on the case $r=1$ by rescaling. By the uniform interior $W^{m,p}$ estimates derived by the authors in \cite[Theorem 1.3]{nsx}, we have
 \begin{align*}
\Big(\fint_{B(x, t)} |\na^m u_\va|^p  \Big)^{1/p}\leq C \Big(\fint_{B(x, 2t)} |\na^m u_\va|^2dx\Big)^{1/2},
\end{align*}
 whenever $u_\va$ is a weak solution to $\mc{L}_\va u_\va=0$ in $B(x,2t)$.
Therefore, in view of (\ref{tholdere1}) and (\ref{remark21}),  we have for any $0<\lm<1$ and $y\in B(x_0,1)\cap \Om$,
\begin{align}\label{pl411}
\begin{split}
\Big(\fint_{B(y,\de(y)/8)} |\na^m u_\va|^p  \Big)^{1/p}&\leq C \Big(\fint_{B(y,\de(y)/4)} |\na^m u_\va|^2\Big)^{1/2}\\
&\leq C [\de(y)]^{\lm -1}\Big(\int_{B(x_0,2)\cap \Om}|\na^m u_\va|^2\Big)^{1/2},
\end{split}
\end{align}
where $\de(y)$ denotes the distance of $y$ to $\pa(B(x_0,2)\cap\Om)$. Fix $\lm\in (1-1/p, 1)$ and integrate (\ref{pl411}) with  respect to $y$ in $B(x_0,1)\cap \Om$. We obtain that
\begin{align}\label{pl412}
\int_{B(x_0,1)\cap \Om}\fint_{B(y,\de(y)/8)} |\na^m u_\va|^pdx\, dy\leq C  \|\na^m u_\va\|^p_{L^2(B(x_0,2)\cap \Om)}.
\end{align}
 We then deduce from Fubini's theorem that,
\begin{align}\label{pl413}
\int_{B(x_0,1)\cap \Om}|\na^m u_\va(x)|^p   \int_{\left\{y\in B(x_0,1)\cap \Om:\, |x-y|<\de(y)/8\right\} } \frac{1}{\de(y)^d} dy\, dx\leq C \|\na^m u_\va\|^p_{L^2(B(x_0,2)\cap \Om)}.
\end{align}
Observe that when $|x-y|<\de(y)/8$, it holds that
\begin{align}\label{pl414}
\frac{1}{2}\de(y)\leq \de(x) \leq 2\de(y).
\end{align}
We thus conclude that
$$ B(x_0,1)\cap \Om \cap B(x, \de(x)/16)\subset \big\{y\in B(x_0,1)\cap \Om: |x-y|<\de(y)/8\big\} $$ for any $x\in B(x_0,1)\cap \Om.$
 This, together with (\ref{pl414}), implies that
$$ \int_{\{y\in B(x_0,1)\cap \Om: |x-y|<\de(y)/8\} } \frac{1}{\de(y)^d}dy  \geq C_0>0.$$
Taking this into (\ref{pl413}), we obtain (\ref{l411}) immediately.
\end{proof}

With Lemma \ref{lemm4.1} at our disposal, we are ready to prove Theorem \ref{twmp}. The proof is based on a real-variable argument in the following theorem, which is formulated in \cite{shenan2005,shenad2007}.
\begin{theorem} \label{szwmp}
Let $q>2$ and $\Om$ be a bounded Lipschitz domain. Let $F\in L^2(\Om)$ and $f\in L^p(\Om)$ for some $2<p<q<\infty$. Suppose that for each ball $B\subset\R^d$ with the property that $|B|<c_0|\Om|$, and either $4B\subset\Om$ or $B$ is centered on $\pa\Om$, there exists two measurable functions $F_B$ and $R_B$ on $2B\cap\Om$ such that
\begin{align}
&|F|\leq |F_B|+|R_B| \text{ on } 2B\cap\Om,\label{tszwmp1} \\
&\Big(\fint_{2B\cap\Om} |R_B|^q\Big)^{1/q}\leq C_1 \Big\{ \Big(\fint_{4B\cap\Om} |F|^2\Big)^{1/2} +\sup_{B\subset B' \subset 4B_0} \Big(\fint_{B'\cap\Om} |f|^2\Big)^{1/2}\Big\},\label{tszwmp2} \\
  &\Big(\fint_{2B\cap\Om} |F_B|^2\Big)^{1/2} \leq C_2 \sup_{B\subset B' \subset 4B_0} \Big(\fint_{B'\cap\Om} |f|^2\Big)^{1/2}+\de\Big(\int_{4B\cap\Om} |F|^2\Big)^{1/2},\label{tszwmp3}
\end{align}
where $C_1,C_2>0, 0<c_0<1.$  Then there exists $\de_0>0$, depending only on $C_1,C_2,c_0, p,q$ and the Lipschitz character of $\Om,$ such that, for any $0<\de<\de_0$, $F\in L^p(\Omega)$ and
\begin{align} \label{tszwmp4}
\Big(\fint_{\Om} |F|^p\Big)^{1/p} \leq C   \Big(\fint_{\Om} |F|^2\Big)^{1/2}+ \Big(\fint_{\Om} |f|^p\Big)^{1/p},
\end{align}
where $C$ depends only on $d, C_1,C_2, c_0, p,q$ and the Lipschitz character of $\Om$.
\end{theorem}

\begin{proof}[\bf Proof of Theorem \ref{twmp}.] Since the desired estimate is trivial when $p=2$, it suffices  to consider the case $p>2.$ Thanks to the extension theorem in \cite[p.223]{mms}, for any $\dot{g}=\{g_\ga\}_{|\ga|\leq m-1}\in \dot{B}^{m-1/p}_p(\pa\Om)$ there exist a $G\in W^{m,p}(\Om)$ such that $$ Tr(D^\ga G)=g_\ga \quad\text{for } 0\leq |\ga|\leq m-1, \quad\quad   \|G\|_{ W^{m,p}(\Om)}\leq C \|\dot{g}\|_{\dot{B}^{m-1/p}_p(\pa\Om)}. $$  Therefore, we can restrict our investigations to the problem with homogeneous boundary conditions.
\begin{align*}
\mc{L}_\va \overline{u}_\va=\sum_{|\al|\leq m}D^\al \overline{f}^\al \quad\text{ in } \Om, \quad\quad
 Tr(D^\ga \overline{u}_\va)=0 \quad\text{ on } \pa\Om \quad\text{for}~0\leq |\ga|\leq m-1,
\end{align*}
where $\overline{u}_\va=u_\va-G$ and
 \begin{align*}
 \overline{f}^\al=f^\al+(-1)^{m+1}\sum_{|\be|=m}A^{\al\be} D^\be G   \quad\text{ for }  |\al|=m , \quad\text{and}\quad
 \overline{f}^\al=f^\al  \quad\text{ for } |\al|< m .
 \end{align*}
Let $F=|\na^m\overline{u}_\va|$ and $f(x)=\sum_{|\al|\leq m}|\overline{f}^\al|$. We only need to construct the functions $F_B$, $R_B$ and then verify the conditions (\ref{tszwmp1}),   (\ref{tszwmp2}) and (\ref{tszwmp3}) to hold for balls $B(x_0,r)$ with the property $|B|<c_0|\Om|$ and either $4B\subset\Om$ or $B$ is centered on $\pa\Om$. The case of $4B\subset\Om$ has been investigated  for interior $W^{m,p}$ estimates in \cite{nsx}. So here we only consider the situation that $B$ is centered on $\pa\Om$.

Let $B=B(x_0,r)$ for some $x_0\in \pa \Om$ and $0<r<r_0/16.$ Let $v_\va \in H_0^{m}(4B\cap\Om; \R^n)$ be the solution to  $\mc{L}_\va v_\va=\sum_{|\al|\leq m}D^\al \overline{f}^\al $ in  $4B\cap\Om $ and set $$F_B=|\na^m v_\va|,\quad R_B=|\na^m w_\va|,\quad w_\va=\overline{u}_\va-v_\va.$$
Then it is obvious that
\begin{align*}
 &|F|\leq |F_B|+|R_B| \text{ on } 2B\cap\Om,\\
  &\Big(\fint_{2B\cap\Om} |F_B|^2\Big)^{1/2}\leq C\Big(\fint_{4B\cap\Om} |\na^m v_\va|^2\Big)^{1/2}  \leq C  \Big(\fint_{4B\cap\Om} |f|^2\Big)^{1/2},
\end{align*}
which imply the conditions (\ref{tszwmp1}) and (\ref{tszwmp3}). Furthermore, note that $$\mc{L}_\va w_\va=0 \text{ in } 4B\cap\Om,~~~~~~~Tr(D^\ga w_\va)=0  \quad\text{on }  4B \cap\pa\Om  \quad\text{for } 0 \leq |\ga|\leq m-1.  $$  By Lemma \ref{lemm4.1}, we know that for any $2<p<\infty$,
\begin{align*}
\Big(\fint_{2B\cap\Om} |\na^m w_\va|^p\Big)^{1/p} & \leq  C \Big(\fint_{4B\cap \Om } |\na^m w_\va |^2\Big)^{1/2}\nonumber\\
& \leq C \Big(\fint_{4B\cap \Om } |\na^m \overline{u}_\va |^2\Big)^{1/2} +C \Big(\fint_{4B\cap \Om } |\na^m v_\va |^2\Big)^{1/2}\nonumber\\
&\leq C \Big(\fint_{4B\cap \Om } |\na^m \overline{u}_\va |^2\Big)^{1/2} +C  \Big(\fint_{4B\cap\Om} |f|^2\Big)^{1/2},
\end{align*}
which implies (\ref{tszwmp2}). Noticing that all the conditions in Theorem \ref{szwmp} are verified, (\ref{twm11}) follows from (\ref{tszwmp4}) immediately.
\end{proof}
Note that if $u_\va \in W_0^{m,p}(\Om, \R^n)$ is a weak solution to $\mc{L}_\va u_\va=f$ in $\Om$, and $u_\va^*\in W_0^{m,p'}(\Om, \R^n)$ is a weak solution to $\mc{L}^*_\va u^*_\va=    f^{*}$ in $\Om$, where $p'=p/(p-1)$. Then we have
\begin{align*}
 \langle   f , u^*_\va \rangle_{W^{-m,p}\times W_0^{m,p'}}=\sum_{|\al|=|\be|=m}\int_\Om  A^{\al\be}(x/\va)D^\be u_\va D^\al u^*_\va=\langle   f^{*}, u_\va \rangle_{W^{-m,p'}\times W_0^{m,p}}.
\end{align*}
Therefore,  Theorem \ref{twmp} also holds for  $1\!<\!p\!<\!2$ by a standard duality argument.

As a consequence of Theorem \ref{twmp}, one can obtain $C^{m-1,\lm}$ estimate on $u_\va$ in $\Om$ immediately. However, we choose to provide a local version using the localization argument mentioned in (\ref{pl322}) and  (\ref{pc312}) (where we did not provide any details). The result  will also provide a direct comparison to Theorem \ref{thholder} as well as Corollary \ref{co3.1}.
\begin{corollary}\label{co4.1}
Suppose that $\Omega$ is a bounded $C^1$ domain in $\mb{R}^d$, $ A\in\!V\!M\!O(\R^d)$  satisfies (\ref{cod1})--(\ref{cod3}). For any $x_0\in \pa\Om$ and $0<r \leq c_0$, let $u_\varepsilon\in H^m(\Om\cap B(x_0,4r); \R^n)$ be a weak solution to
$$\mc{L}_\va u_\va= \sum_{|\ze|\leq m} D^\ze f^\ze\quad\text{in } \Om\cap B(x_0,4r), \quad \quad Tr(D^\ga u_\va)=D^\ga G \quad\text{on } \pa\Om\cap B(x_0,4r),\, 0 \leq|\ga|\leq m-1,$$
where $f^\ze\in L^p(\Om\cap B(x_0,4r); \R^n)$ for all $|\ze|\leq m$, and $ G \in W^{m,p}(\Om\cap B(x_0,4r); \R^n) $ with $p> d$ and $p\geq 2$.
Then for any $x,y\in \Om\cap B(x_0,r)$,
\begin{align}
&|\na^{m-1} u_\varepsilon(x)-\na^{m-1} u_\varepsilon(y)|
\leq\, C\frac{|x-y|^\lm}{r^{m-1+\lm}}  \Big\{\Big(\fint_{\Om\cap B(x_0,4r)}|u_\varepsilon|^2\Big)^{1/2} + r^{m-d/p} \|G\|_{W^{m,p}(\Om\cap B(x_0,4r))}\nonumber \\
&\quad\quad\quad\quad\quad\quad\quad\quad\quad\quad\quad\quad  +\sum_{|\zeta|\leq m} r^{2m-|\zeta|}\Big(\fint_{\Om\cap B(x_0,4r)}|f^\zeta|^p\Big)^{1/p} \Big\} ,\label{co410}\\
&\|\nabla^k u_\varepsilon\|_{L^\infty(\Om\cap B(x_0,r))}\leq C \,  r^{-k} \Big\{\Big(\fint_{\Om\cap B(x_0,4r)}|u_\varepsilon|^2\Big)^{1/2}+ r^{m-d/p} \|G\|_{W^{m,p}(\Om\cap B(x_0,4r))} \nonumber\\
&\quad\quad\quad\quad\quad\quad\quad\quad\quad\,\,+ \sum_{|\zeta|\leq m} r^{2m-|\zeta|}\Big(\fint_{\Om\cap B(x_0,4r)}|f^\zeta|^p\Big)^{1/p} \Big\} ,\label{co411}
\end{align}   where $0\leq k \leq m-1, \lm=1-d/p,$
 and $C$ depends only on $d, n,m, p, \mu, \Om$ as well as $\varrho(t)$ in (\ref{vmo}).
 \end{corollary}
\begin{proof} 
By rescaling and translation, we may assume that $r=1, x_0=0.$ Denote $B(0,r)$ as $B_r$ and
let $\wt{D}$ be a $C^1$ domain such that $B_{3/2}\cap \Om\subset \wt{D}\subset B_2\cap \Om$. Set $v_\va=u_\va-G$. It is obvious that $v_\va$ satisfies
\begin{align}\label{plco411}
\begin{split}
&\mc{L}_\va v_\va=\sum_{|\al|=m} D^\al \{A_\va^{\al\be} D^\be G\} + \sum_{|\ze|\leq m} D^\ze f^\ze\, \quad\text{ in } \Om\cap B_4, \\
& Tr(D^\ga v_\va)=0 \,\quad\text{ on } \pa\Om\cap B_4 \quad\text{for } 0<|\ga|\leq m-1.\end{split}\end{align}
Let $\phi \in C_c^\infty(B_{3/2})$ with $\phi=1$ in $B_1$ and $|\na^k\phi|\leq C2^k$. We have
\begin{align*}
&\mc{L}_\va (v_\va \phi)=\sum_{|\al|=|\be|=m}\Big\{ D^\al \big\{A_\va^{\al\be} D^\be G\big\}\phi +\sum_{\al'+\al''=\al, |\al''|\geq 1} C(\al')D^{\al'}\big\{A_\va^{\al\be} D^\be v_\va \big\} D^{\al''} \phi\\
 &\quad\quad\quad\quad\quad\quad\quad + \sum_{\be'+\be''=\be, |\be''|\geq 1}  C(\be') D^\al \big\{A_\va^{\al\be} D^{\be'} v_\va D^{\be''}\phi\big\} \Big\}+ \sum_{|\ze|\leq m} D^\ze f^\ze \phi    \quad \text{ in } \wt{D},\\
 &Tr(D^\ga (v_\va \phi))=0,  \quad\text{ on } \pa\wt{D}  \quad\text{for}\quad  0<|\ga|\leq m-1.
\end{align*}
Observe that for $0\leq \ell=|\al'|\leq  m-1 $, $$   D^{\al'} \big\{A_\va^{\al\be} D^\be v_\va \big\} D^{\al''} \phi  \in  W^{-m,p}(\wt{D}) \quad \text{ if }\,  \na^m v_\va \in L^{q_\ell}(\wt{D}) \text{ with } q_\ell= \frac{dp}{(m-\ell)p+d}~(<p).$$
 Thus, we may deduce from Theorem \ref{twmp} that
\begin{align}\label{pcoco}
\|v_\va\|_{W^{m,p}(B_1\cap \Om)}
  &\leq C  \Big(\int_{\wt{D}} |\na^m G|^p \Big)^{1/p} + C\sum_{0\leq k\leq m-1} \Big(\int_{\wt{D}} |\na^k v_\va|^p \Big)^{1/p} \nonumber\\
  &\quad\quad\quad\quad  + C\sum_{0\leq \ell\leq m-1} \Big(\int_{\wt{D}} |\na^m v_\va|^{q_\ell} \Big)^{1/{q_\ell}}+  C\sum_{|\ze|\leq m}\Big( \int_{\wt{D}} |f^\ze|^p \Big )^{1/p}\nonumber\\
&\leq C  \Big(\int_{B_2\cap \Om} |\na^m G|^p \Big)^{1/p} + C\sum_{0\leq k\leq m-1} \Big(\int_{B_2\cap \Om} |\na^k v_\va|^p \Big)^{1/p} \nonumber\\
&\quad\quad\quad\quad+ C  \sum_{0\leq \ell\leq m-1} \Big(\int_{B_2\cap \Om} |\na^m v_\va|^{q_\ell} \Big)^{1/{q_\ell}} + C\sum_{|\ze|\leq m} \Big(\int_{B_2\cap \Om} |f^\ze|^p \Big)^{1/p}.
\end{align}
Let $p_1= dp/(d+p).$ Thanks to the Poincar\'{e} inequality and Sobolev imbedding,  we have
\begin{align*}
&\|\na^k v_\va\|_{L^p(B_2\cap \Om)}\leq C  \|\na^m v_\va\|_{L^{p_1}(B_2\cap \Om)} \quad \text{for } 0\leq k\leq m-1,\\
&\|\na^m v_\va\|_{L^{q_\ell}(B_2\cap \Om)}\leq C  \|\na^m v_\va\|_{L^{p_1}(B_2\cap \Om)}  \quad \text{for } q_\ell=\frac{dp}{(m-\ell)p+d},\, 0\leq \ell\leq m-1,
\end{align*}
which, combined with (\ref{pcoco}), implies that
\begin{align}
& \|v_\va\|_{W^{m,p}(B_1\cap \Om)}  \leq C   \|\na^m G\|_{L^p(B_2\cap \Om)}  +  C  \|\na^m v_\va\|_{L^{p_1}(B_2\cap \Om)} +C\sum_{|\ze|\leq m} \|f^\ze\|_{L^p(B_2\cap \Om)}.
\end{align}
If $p_1>2$, we can  perform a bootstrap argument for finite times to obtain that
  \begin{align*}
\|v_\va\|_{W^{m,p}(B_1\cap \Om)}
 \leq C \|\na^m G\|_{L^p(B_3\cap \Om)}  + C  \|\na^m v_\va\|_{L^{2}(B_3\cap \Om)}+C\sum_{|\ze|\leq m} \|f^\ze\|_{L^p(B_3\cap \Om)}.
\end{align*}
 By Caccioppoli's inequality, this implies that
 \begin{align*}
\|u_\va\|_{W^{m,p}(B_1\cap \Om)}
&\leq C \Big\{  \|\na^m  v_\va\|_{L^2(B_3\cap \Om)} + \| G\|_{W^{m,p}( B_4\cap \Om)}+  \sum_{|\ze|\leq m} \|f^\ze\|_{L^p(B_4\cap \Om)}\Big\}\\
& \leq C \Big\{ \|u_\va\|_{L^2(B_4\cap \Om)}+ \| G\|_{W^{m,p}( B_4\cap \Om)}+\sum_{|\ze|\leq m} \|f^\ze\|_{L^p(B_4\cap \Om)} \Big\},
\end{align*}
which gives (\ref{co410}) and (\ref{co411})  by Sobolev imbedding.
\end{proof}

\section{Uniform $C^{m-1,1}$ estimates}
In this section, we consider uniform boundary $C^{m-1, 1}$ estimates for $u_\va$ in $C^{1,\te} (0<\te<1)$ domains.
Throughout the section, we always assume that $A$ satisfies (\ref{cod1}) and (\ref{cod3}).
Similar to Section 4, we only need to consider equations in $(D_r, \De_r)$ defined as in (\ref{drder}) with the defining function $\psi\in C^{1,\te}(\R^{d-1})$ satisfying  $ \psi(0)=0,  \|\na\psi\|_{C^\te(\R^{d-1})}\leq M_1$.

Let $u_\va\in H^m(D_{2}; \mb{R}^n )$ be a weak solution to
\begin{align*}
 \mc{L}_\va u_\va= \sum_{|\al|\leq m-1}D^\al f^\al    \quad\text{ in } D_1, \quad\quad
 Tr (D^\ga u_\va)=D^\ga G  \quad\text{ on }  \De_1  \quad\text{for  } 0\leq|\ga|\leq m-1,
\end{align*}
where $f^\al\in L^{q}(D_1; \R^n) $ with $ q>d, q\geq 2$, and $ G\in C^{m,\sg}(D_1; \R^n)$ for some $0<\sg\leq\te$.
 For $0<r\leq 1$, define the following auxiliary quantities,
\begin{align}
 \Phi(r,u_\va)=
\frac{1}{r^m}\inf_{P_{m-1}\in \mf{P}_{m-1}}\Big\{&\Big(\fint_{D_{r}}|u_\va-P_{m-1} |^2 \Big)^{1/2}
  + \sum_{|\al|\leq m-1}r^{2m-|\al|} \Big(\fint_{D_r}|f^\al|^q\Big)^{1/q}\nonumber\\ &+ \sum_{j=0}^{m} r^j   \|\na^j (G-P_{m-1})\|_{L^\infty(D_r)}
  \Big\},\label{phi}\end{align}
  \begin{align}
H(r; u_\va)=
\frac{1}{r^{m}}\inf_{P_{m}\in \mf{P}_{m}} \Big\{&\Big(\fint_{D_r}|u_\va-P_{m}|^2 \Big)^{1/2}+ \sum_{|\al|\leq m-1}r^{2m-|\al|} \Big(\fint_{D_r}|f^\al|^q\Big)^{1/q}
 \nonumber\\ &+  \sum_{j=0}^{m} r^j   \|\na^j (G-P_{m})\|_{L^\infty(D_r)}
+r^{m+\sg}\|\na^m (G -P_{m})\|_{C^{0,\sg}(D_r)} \Big\} .\label{hru}
\end{align}

\begin{lemma}\label{lemm5.1}
For $0<\va\leq r\leq1$, let $\Phi(r; u_\va)$ be defined as in (\ref{phi}). Then there exists $u_0\in H^m(D_r; \mb{R}^n )$ such that
$\mc{L}_0 u_0= \sum_{|\al|\leq m-1}D^\al f^\al$ in $D_r$, $Tr (D^\ga u_0)=D^\ga G $ on $\De_r$ for $ 0\leq|\ga|\leq m-1$, and
\begin{align}\label{l511}
\frac{1}{r^m}\Big(\fint_{D_r}|u_\va-u_0|^2 \Big)^{1/2}\leq C \Big(\frac{\va}{r}\Big)^{1/4} \Phi(2r; u_\va),
\end{align}
where  $C$  depends only on $d, n,m,  q, \sigma, \mu$ and $M$ in (\ref{c1c}).
\end{lemma}
\begin{proof}
The proof is the same as Lemma \ref{lemm3.1}, we therefore omit the details.
\end{proof}

\begin{lemma}\label{lemm5.2}
 Let $u_0\in H^m(D_{r}; \mb{R}^n )$ be a weak solution to $\mc{L}_0 u_0= \sum_{|\al|\leq m-1} D^\al f^\al$ in $D_{r}$ with $Tr (D^\ga u_0)=D^\ga G $ on $\De_{r}$ for $ 0\leq|\ga|\leq m-1$. Then there exists a $\de\in (0, 1/4)$, depending only on $d,n,m, q, \sigma,\mu,\te$ and $ M_1$ in (\ref{c2c}), such that
\begin{align}\label{l521}
H(\de r;u_0)\leq \frac{1}{2} H (r;u_0).
\end{align}
\end{lemma}
\begin{proof}
The proof, parallel to that of Lemma \ref{lemm3.2}, is mainly based on $C^{m,\sg}$ estimates for higher order  elliptic systems with constant coefficients in $C^{1,\te} (0<\sg\leq \te)$ domains.
By rescaling, we assume that $r=1$.  Taking $$P_{m}(x)= \sum_{|\al|\leq m}  \frac{1}{\al!} D^\al u_0(0) x^{\al}=\sum_{|\al|\leq m}  \frac{1}{\al!} D^\al G(0) x^{\al},$$  it is not difficult to find that for any $0< \de<1/4$ and  any $0<\sg'< \min\{1-d/q, \sg \}$,
\begin{align}\label{pl521}
 H(\de,u_0)\leq C \de^{\sg'} \|u_0\|_{C^{m,\sg'}(D_\de)}&+ C \de^{m-|\al|-d/q} \sum_{|\al|\leq m-1}  \Big( \fint_{D_{1}} |f^\al|^{q}\Big)^{1/q}
  + C \de^{\sg} \|G\|_{C^{m,\sg}(D_1)}.
\end{align}
By the localization argument and the $C^{m,\sg}$ estimate for higher order elliptic systems with constant coefficients (see e.g., \cite[Corollary 2.4]{miyazaki2013}),
we have
\begin{align}\label{pl522}
\|u_0\|_{C^{m,\sg'}(D_\de)}&\leq C\|u_0\|_{C^{m,\sg'}(D_{1/4})}\nonumber\\
&\leq C\Big\{ \Big(\fint_{D_1} |u_0|^2\Big)^{1/2}
+ \sum_{|\al|\leq m-1}\Big(\fint_{D_1}|f^\al|^q\Big)^{1/q}+ \|G \|_{C^{m,\sg}(D_1)}  \Big\}
\end{align} for  $0<\sg'< \min\{ 1-d/q, \sg \}.$
 Taking (\ref{pl522}) into (\ref{pl521}) and setting  $\de$ small enough, we get
\begin{align*}
H(\de,u_0)\leq \frac{1}{2} \Big\{ \Big(\fint_{D_1} |u_0|^2\Big)^{1/2}   + \sum_{|\al|\leq m-1}\Big(\fint_{D_1}|f^\al|^q\Big)^{1/q}+\|G \|_{C^{m,\sg}(D_1)}\Big\}.
\end{align*}
For any $P_{m}\in\mf{P}_{m}$, substituting $u_0, G$ by $u_0-P_{m}$ and $ G-P_{m}$ respectively and taking the infimum, we obtain (\ref{l521}) immediately.
\end{proof}

\begin{lemma}\label{lemm5.3}
Let $0<\va<1/2$ and   $\Phi(r; u_\va), H(r; u_\va)$ be defined as in (\ref{phi}) and (\ref{hru}).  Let $\de$ be given by Lemma \ref{lemm5.2}. Then for any $r\in [\va,1/2],$
\begin{align}\label{l531}
 H(\de r; u_\va)\leq \frac{1}{2} H(r; u_\va)+ C\Big( \frac{\va}{r}\Big)^{1/4} \Phi(2r; u_\va),
\end{align}
where $C$ depends only on $d,n, m, q,\mu,\sg,\te$ and $M_1$ in (\ref{c2c}).
\end{lemma}
\begin{proof} Similar to Lemma \ref{lemm3.3}, the result follows from Lemmas \ref{lemm5.1} and \ref{lemm5.2}. We thus omit the details.
\end{proof}

\begin{lemma}\label{lemm5.4}
Let $H(r)$ and $h(r)$ be two nonnegative continuous functions on the interval $(0,1],$ and let $\varepsilon \in (0,1/4).$ Assume that
\begin{align}\label{l44con1}
\max_{r\leq t\leq 2r} H(t)\leq C_0H(2r), ~~~~~\max_{r\leq t,s\leq 2r} | h(t)-h(s)|\leq C_0H(2r),
\end{align}
for any $r\in [\varepsilon, 1/2],$  and also
\begin{align}\label{l44con2}
H(\delta r) \leq \frac{1}{2} H(r) + C_0 \omega (\varepsilon/r)\big\{ H(2r)+h(2r)\big\},
\end{align}
for any $r\in [\varepsilon, 1/2],$ where  $\delta\in (0,1/4)$  and $\omega$ is a nonnegative increasing function on $[0,1]$ such that
$\omega(0)=0$ and
 $
\int_0^1  \omega(\varsigma)/\varsigma \, d\varsigma< \infty.
$
Then
\begin{align}\label{l44re1}
\max_{\varepsilon\leq r\leq 1} \big\{H(r) +h(r)\big\}\leq C \big\{H(1) +h(1)\big\}.
\end{align}
\end{lemma}
\begin{proof}
See Lemma 8.5 in \cite{shenan2017}.
\end{proof}
Armed with  lemmas above, we are ready to prove Theorem \ref{tlip}.
\begin{proof}[\bf Proof of Theorem \ref{tlip}]
We assume that  $0<\varepsilon\leq r< 1/4$, since if else (\ref{tlipre1}) is just a consequence of Caccioppoli's inequality.
Let $u_\va\in H^m(D_{1}; \mb{R}^n )$ be a weak solution to $$
 \mc{L}_\va u_\va= \sum_{|\al|\leq m-1}D^\al f^\al  \quad \text{in } D_1,   \quad\quad
 Tr (D^\ga u_\va)=D^\ga G  \quad\text{on }  \De_1  \quad\text{for  } 0\leq|\ga|\leq m-1,
$$
where $f^\al\in L^{q}(D_1; \R^n) $ with $  q>d, q\geq2$, and $ G\in C^{m,\sg}(D_1; \R^n)$ for some $0<\sg\leq \te$.
For  $r\in (0,1)$, let $H(r)= H(r, u_\va), \Phi(r)=\Phi(r,u_\va)$, and $\omega(y)=y^{1/4}$.
Define  $$ h(r)=  \sum_{|\alpha|=m}\frac{1}{\alpha!} |D^\alpha P_{mr}(x)|,$$   where
$P_{mr}\in  \mathfrak{P}_{m}$ such that
\begin{align}\label{ptlip00}
H(r)=\frac{1}{r^{m}} \Big\{&\Big(\fint_{D_r}|u_\va-P_{mr}|^2 \Big)^{1/2} + \sum_{|\al|\leq m-1}r^{2m-|\al|} \Big(\fint_{D_r}|f^\al|^q\Big)^{1/q}
\nonumber\\ &\quad\quad+  \sum_{j=0}^{m} r^j   \|\na^j (G-P_{mr})\|_{L^\infty(D_r)}
+r^{m+\sg}\|\na^m (G -P_{mr})\|_{C^{0,\sg}(D_r)} \Big\}.
\end{align}
Next let us check that $H(r), h(r)$ satisfy  conditions (\ref{l44con1}) and (\ref{l44con2}).
 From the definition it is obvious that
\begin{align}\label{ptlip01}
H(t)\leq  C H(2r)\quad \text{for any } t\in [r,2r].\end{align}
On the other hand, by the definition of $h(r),$
\begin{align}\label{ptlip02}
|h(t)-h(s)| &\leq \sum_{|\alpha|=m} \frac{1}{\alpha!} |D^\alpha (P_{mt}-P_{ms})|
  =  \sum_{|\alpha|=m} \frac{1}{\alpha!} \Big(\fint_{D_r} |D^\alpha (P_{mt}-P_{ms}) |^2 \Big)^{1/2}\nonumber \\
 &\leq C\Big(\fint_{D_{t}} | \na^m( G-P_{mt} )  |^2 \Big)^{1/2}+ C\Big(\fint_{D_{s}} | \na^m(G-P_{ms} )  |^2 \Big)^{1/2}\nonumber\\
  &\leq C \big\{ H(t) +H(s) \big\}
 \leq C H(2r),
\end{align}
where we have used the fact $r\leq t,s\leq 2r$, the definition of $P_{mr}$  and (\ref{ptlip01}) respectively for the last three inequalities. Combining (\ref{ptlip01}) with (\ref{ptlip02}), we know that condition (\ref{l44con1}) is satisfied.  Finally, from the definitions of $\Phi(r), H(r) $ and $h(r)$, we obtain that
\begin{align*}
\Phi(r)& \leq  \frac{1}{r^m}\Big\{ \Big(\fint_{D_{r}}|u_\va-P_{mr} |^2 \Big)^{1/2}
  + \sum_{|\al|\leq m-1}r^{2m-|\al|} \Big(\fint_{D_r}|f^\al|^q\Big)^{1/q}
  + \sum_{j=0}^{m} r^j   \|\na^j (G-P_{mr})\|_{L^\infty(D_r)}
  \Big\}\nonumber\\
  &\quad\quad\quad+ \inf_{P_{m-1}\in \mf{P}_{m-1}}\frac{1}{r^m}\Big\{ \Big(\fint_{D_{r}}|P_{mr}-P_{m-1} |^2 \Big)^{1/2} + \sum_{j=0}^{m} r^j   \|\na^j (P_{mr}-P_{m-1})\|_{L^\infty(D_r)}\Big\}\nonumber\\
&\leq  H(r)+ C h(r) ,
\end{align*}
which, together with (\ref{l531}), implies (\ref{l44con2}). Note that all conditions of Lemma \ref{lemm5.4} are verified. Therefore, for all $\va\leq r\leq 1,$
\begin{align}\label{ptlip04}
  \frac{1}{r^m} \inf_{P_{m-1}\in \mf{P}_{m-1}}   \Big(\fint_{D_r} |u_\varepsilon -P_{m-1}|^2\Big)^{1/2}
  \leq  \Phi(r)  \leq C  \big\{ H(r)+h(r)  \big\} \leq C \big\{ H(1)+h(1)  \big\}.
\end{align}
From the definition of $H(1)$, we have
\begin{align}\label{ptlip05}
h(1)\leq  \sum_{|\al|=m} \Big(\fint_{D_1} |D^\al (G- P_{m1}) |^2\Big)^{1/2} + C \|\na^m G\|_{L^\infty(D_1)} \leq  C \big\{H(1) +  \|\na^m G\|_{L^\infty(D_1)} \big\} .\end{align}
It then follows that
\begin{align*}
 \frac{1}{r^m}\!\inf_{P_{m-1}\in \mf{P}_{m-1}}\!\Big(\fint_{D_r} |u_\varepsilon\!-\!P_{m-1} |^2\Big)^{1/2}\! \leq C  \Big\{\Big(\fint_{D_1}|u_\va|^2 \Big)^{1/2}
     +\!\sum_{|\al|\leq m-1}\Big(\fint_{D_1}|f^\al|^q\Big)^{1/q} + \|G \|_{C^{m,\sg}(D_1)}\Big\},
\end{align*}
which gives (\ref{tlipre1}) through  Caccioppoli's inequality.
\end{proof}

\begin{corollary}\label{co5.1}
In addition to the assumptions of Theorem \ref{tlip}, if $A$ satisfies (\ref{hol}),
then
\begin{align}\label{co511}
\|\na^mu_\va\|_{L^{\infty}(D_{1/4})}\leq C  \Big\{\Big(\fint_{D_1}|u_\va|^2 \Big)^{1/2}     +\sum_{|\al|\leq m-1}\Big(\fint_{D_1}|f^\al|^q\Big)^{1/q} + \|G \|_{C^{m,\sg}(D_1)}\Big\},
\end{align}
where $C$ depends only on $d,n,m,q,  \sg,\mu$ as well as $\Lambda_0, \tau_0$ in (\ref{hol}) and $\te, M_1$ in (\ref{c2c}).
\end{corollary}
\begin{proof}
 It is enough to consider the case $0<\va<1/2$, since otherwise the coefficient is uniformly H\"{o}lder continuous and the result (\ref{co511}) is known, see e.g., \cite[Corollary 2.4]{miyazaki2013}.
Setting $$v_\va(x)=u_\va(\va x)-\wt{G}(x),  \quad\,\wt{G}(x)=  G(\va x), \quad \,  \wt{f}^\al(x)=\va^{2m-|\al|} f(\va x),$$   we have
\begin{equation}\label{618}
\begin{cases}
\mc{L}_1 v_\va =   \sum_{|\al|\leq m-1}D^\al \wt{f}^\al(x)+  \sum_{|\al|=|\beta|=m}D^\al \left\{ A^{\al\be}D^\be \wt{G}(x)\right\}  \quad\text{ in } D_1,  \vspace{0.2cm} \\
 Tr (D^\ga  v_\va)=0, \quad\quad\quad \text{on }  \De_1\quad\text{for } 0\leq|\ga|\leq m-1.
 \end{cases}
\end{equation}
Let  $\phi \in C_c^\infty(B_1)$ with $\phi=1$ in $B_{1/4}$ and $|\na^k\phi|\leq C2^k$, and let $\wt{D}$ be a $C^{1,\te}$ domain such that $D_{1/4}\subseteq \wt{D}\subseteq D_{1/2}$. We have
\begin{align*}
 &\mc{L}_1 (v_\va \phi) = E(x) \phi    + \sum_{\substack{|\al|=|\be|=m\\\ze+\eta=\be\\ |\eta|\geq 1}}C(\ze) D^\al \big\{A^{\al\be} D^\ze v_\va D^\eta\phi\big\}
 +\sum_{\substack{|\al|=|\be|=m\\ \ze'+\eta'=\al\\ |\eta'|\geq 1}} C(\ze') D^{\ze'}\big\{A^{\al\be} D^\be v_\va \big\} D^{\eta'} \phi    ~~\text{ in } \wt{D},    \\
 &Tr (D^\ga  (v_\va \phi))=0 \quad\quad\text{ on }  \pa \wt{D} \quad\text{for } 0\leq|\ga|\leq m-1,
 \end{align*}
where $$E(x)=   \sum_{|\al|\leq m-1}D^\al \wt{f}^\al(x)+ \sum_{|\al|=|\beta|=m}D^\al \left\{A^{\al\be}D^\be \wt{G}(x)\right\}. $$
Thanks to the boundary  $C^{m,\lm}$ estimate for operator $\mc{L}_1$ in $C^{1,\te}$ domains  \cite{miyazaki2013},  we know that
for  $0<s<1/2,$ and for any $q, p>d,$
\begin{align}\label{pco512}
\|\na^m v_\va\|_{L^\infty(D_s)} \leq  C    \Big\{ \Big(\fint_{D_1} |v_\va|^2\Big)^{1/2}  +   \|\wt{G}\|_{C^{m,\sg}(D_1)}
 + \sum_{|\al|\leq m-1}\Big(\fint_{D_1}|\wt{f}^\al|^q\Big)^{1/q}  +  \|v_\va\|_{W^{m,p}(\wt{D})}  \Big\}.
 \end{align}
 Thanks to the $W^{m,p}$ estimate for (\ref{618}), there exists some $p>d $ such that
 \begin{align*}
 \|v_\va\|_{W^{m,p}(\wt{D})}
 &\leq C    \Big\{ \Big(\fint_{D_1} |v_\va|^2\Big)^{1/2}  + \sum_{|\al|\leq m-1}\Big(\fint_{D_1}|\wt{f}^\al|^q\Big)^{1/q}
 +   \|\wt{G}\|_{C^{m,\sg}(D_1)}  \Big\},
 \end{align*}
 which, combined with (\ref{pco512}), implies that
  \begin{align*}
 \|\na^m v_\va\|_{L^\infty(D_s)} \leq C    \Big\{ \Big(\fint_{D_1} |v_\va|^2\Big)^{1/2}  + \sum_{|\al|\leq m-1}\Big(\fint_{D_1}|\wt{f}^\al|^q\Big)^{1/q}
 +   \|\wt{G}\|_{C^{m,\sg}(D_1)}  \Big\}.
 \end{align*}
It then follows form the change of variables that
\begin{align}\label{pco513}
\|\na^m u_\va\|_{L^\infty(D_r)}&\leq C    \frac{1}{\va^m} \Big\{\Big(\fint_{D_\va}| u_\va|^2\Big)^{1/2} + \va^{2m-|\al|}\sum_{|\al|\leq m-1}\Big(\fint_{D_\va}|f^\al|^q\Big)^{1/q}
\nonumber\\
 & \quad+  \sum_{j=0}^{m} \va^j   \|\na^j G \|_{L^\infty(D_\va)}
+\va^{m+\sg}\|\na^m G\|_{C^{0,\sg}(D_\va)}\Big\} \quad\quad\text{for }  0<r<\va/2 .
\end{align}
Using (\ref{ptlip00}), (\ref{ptlip04}) and (\ref{ptlip05}), we may conclude from (\ref{pco513}) that,
\begin{align*}
\|\na^m u_\va\|_{L^\infty(D_r)}&\leq C \big\{H(\va)+h(\va)\big\}\leq C \big\{H(1)+h(1)\big\}\\ &\leq C  \Big\{\Big(\fint_{D_1}|u_\va|^2 \Big)^{1/2}
+ \sum_{|\al|\leq m-1}\Big(\fint_{D_1}|f^\al|^q\Big)^{1/q} + \|G \|_{C^{m,\sg}(D_1)}\Big\}  \quad\text{for } 0\!<\!r<\!\va/2.
\end{align*}
This, together
 with the interior uniform $C^{m-1,1}$ estimate for $u_\va$ derived in \cite[Theorem 1.2]{nsx}, gives (\ref{co511}).
\end{proof}

\vspace{0.8cm}
\noindent\textbf{Acknowledgments.} This paper was completed during the authors'  visits at University of Kentucky. They are much obliged to Professor Zhongwei Shen for the guidance. Our special thanks also go to Professor Russell Brown and the Department of Mathematics for the warm hospitality and support, as well as to Dr. Jinping Zhuge for some enlightening discussions.

\vspace{0.5cm}
\noindent Weisheng Niu,\\
School of Mathematical Science, Anhui University,
Hefei, 230601, P. R. China\\~~E-mail:niuwsh@ahu.edu.cn\vspace{0.5cm}\\
\noindent Yao Xu,\\
Department of Mathematics, Nanjing University,
Nanjing, 200093, P. R. China\\~~E-mail:dg1421012@smail.nju.edu.cn

 \end{document}